\documentclass[A4paper,12pt]{article}

\usepackage{latexsym}
\usepackage{amsmath,amsthm,amsfonts,amssymb}
\usepackage{graphicx}

\usepackage{hyperref}
\usepackage{multicol}
\usepackage{float}

\usepackage{tikz}
\usepackage{tikzscale}

\newcommand{\bZ}{\mathbb{Z}}

\newcommand{\cp}{\,\square\,}

\DeclareMathOperator{\Aut}{Aut}

\DeclareMathOperator{\Cay}{Cay}

\theoremstyle{plain}
\newtheorem{Theorem}{Theorem}
\newtheorem{Corollary}[Theorem]{Corollary}
\newtheorem{Proposition}[Theorem]{Proposition}

\newtheorem{Lemma}[Theorem]{Lemma}

\theoremstyle{remark}

\theoremstyle{definition}

\newcommand{\comment}[1]{}

\begin{document}

\title{Vertex-transitive graphs and their arc-types}

\author{
  Marston Conder,
  Toma\v{z} Pisanski,  and
  Arjana \v{Z}itnik
}

\date{6 May 2015}

\maketitle

\begin{abstract}
Let $X$ be a finite vertex-transitive graph of valency $d$, and let $A$ be the full 
automorphism group of $X$.  Then the {\em arc-type\/} of $X$ is defined in terms 
of the sizes of the orbits of the action of the stabiliser $A_v$ of a 
given vertex $v$ on the set of arcs incident with $v$.  
Specifically, the arc-type is the partition of $d$ as the sum  
$n_1 + n_2 + \dots + n_t + (m_1 + m_1) + (m_2 + m_2) + \dots + (m_s + m_s)$,
where $n_1, n_2, \dots, n_t$ are the sizes of the self-paired orbits, and  
$m_1,m_1, m_2,m_2, \dots, m_s,m_s$ are the sizes of the non-self-paired 
orbits, in descending order.  
In this paper, we find the arc-types of several families of graphs. 
Also we show that the arc-type of a Cartesian product of two `relatively prime'  
graphs is the natural sum of their arc-types. 
Then using these observations, we show that with the exception of $1+1$ and $(1+1)$, 
every partition as defined above is \emph{realisable}, in the sense that there exists 
at least one graph with the given partition as its arc-type.
\medskip




\end{abstract}

\section{Introduction}
\label{sec:intro}
Vertex-transitive graphs hold a significant place in mathematics, dating back to 
the time of first recognition of the Platonic solids, and also now in other disciplines 
where symmetry (and even other properties such as rigidity) play an important role, 
such as fullerene chemistry, and interconnection networks. 

A major class of vertex-transitive graphs is formed by Cayley graphs, which 
represent groups in a very natural way.  (For example, the skeleton of the $C_{60}$ 
molecule is a Cayley graph for the alternating group $A_5$.)  It is easy to test 
whether a given vertex-transitive graph is a Cayley graph for some group: 
by a theorem of Sabidussi \cite{Sabidussi}, this happens if and only if the 
automorphism group of the graph contains a subgroup that acts regularly on vertices. 
Vertex-transitive graphs that fail this test are relatively rare, the Petersen graph 
being a famous example.  A recent study of small vertex-transitive graphs of 
valency 3 in \cite{PSVvt}  shows that among $111360$ such graphs of order 
up to 1280, only $1434$ of them are not Cayley graphs. 

For almost every Cayley graph, the automorphism group itself is a
vertex-regular subgroup $G$ (see \cite{BabaiGodsil}).  Any such graph 
is called a \emph{graphical regular representation} of the group $G$, or briefly, a {\em GRR.} 
In a book devoted to the study of $3$-valent graphs that are GRRs, 
Coxeter, Frucht and Powers \cite{Coxeter} classified vertex-transitive $3$-valent 
graphs into four types, according to the action of the automorphism group on 
the {\em arcs} (ordered edges) of the graph.  One class consists of those which 
are arc-transitive, another of those in which there are two orbits 
on arcs, and the other two are two classes of GRRs. 
Arc-transitive graphs are also called {\em symmetric.} 

Symmetric graphs have been studied quite intensively, especially in the $3$-valent case, 
by considering the action of the automorphism group on non-reversing walks of given length $s$, 
known as {\em s-arcs.} 
For example, it was shown by Tutte \cite{Tutte0,Tutte1} 
that every finite symmetric $3$-valent graph is $s$-arc-regular for some $s \le 5$, 
and hence that order of the stabiliser of a vertex in the automorphism group of 
every such graph is bounded above by $48$.  Tutte's theorem and related work 
have been used to determine all symmetric $3$-valent graphs on up to $10,000$ 
vertices; see \cite{Foster,ConderFosterCensus,Conder10000}.  
Also Tutte's seminal theorem was generalised much later by Weiss, who used the classification 
of doubly-transitive permutation groups to prove that every finite symmetric graph of valency 
greater than $2$ is $s$-arc-transitive but not $(s+1)$-arc-transitive for some $s \le 7$, 
and in particular, that there are no $8$-arc-transitive finite graphs; see \cite{Weiss}. 

Another important class of vertex-transitive graphs was investigated by Tutte \cite{Tutte} 
and Bouwer \cite{Bouwer}, namely the graphs that are vertex- and edge-transitive 
but not arc-transitive.   These are now called \emph{half-arc-transitive graphs}. 
Every such graph has even valency, and its automorphism group has two orbits on arcs, 
with every arc $(v,w)$ and its reverse $(w,v)$ lying in different orbits; see \cite[p.~59]{Tutte}. 
Bouwer \cite{Bouwer} constructed a family of examples containing one half-arc-transitive graph 
of each even valency greater than $2$, and the first and third author of this paper have recently 
proved that other examples of the  type considered by Bouwer produce infinitely many of 
every such valency; see~\cite{ConderZitnik}. 

According to Coxeter, Frucht and Powers \cite{Coxeter}, the idea of classifying cubic 
vertex-transitive graphs with respect to the arc-orbits of the automorphism group 
originated in some work by Ronald Foster over half a century ago, which was presented 
to his friends in the form of unpublished notes. The classification is carried out rigorously in their 
book \cite{Coxeter}. And although Foster's original idea of `zero-symmetric' graphs was 
later expanded to other valencies under the term GRR, the classification by arc-orbits 
itself was never extended in a systematic way to graphs of other valencies. 
This paper provides a remedy for that omission.
By introducing the concept of {\em arc-type\/}, we provide a language that can 
be used to unify the notions of arc-transitivity and half-arc-transitivity and 
the above-mentioned classification of symmetric $3$-valent graphs, and also 
to extend this classification to vertex-transitive graphs of higher valency.  

We can now define the notion of \emph{arc-type} for a vertex-transitive graph.
Let $X$ be a $d$-valent vertex-transitive graph, with automorphism group $A$.  
We first make a critical observation about the pairing of arc-orbits.  
The orbit of an arc $(v,w)$ under the action of $A$ 
can be \emph{paired} with the orbit of $A$ containing the reverse arc $(w,v)$, and if these 
orbits are the same, then the given orbit is said to be \emph{self-paired}. 
This is analogous with the definition of paired sub-orbits for transitive permutation groups. 
In particular, it applies also to the orbits of the stabiliser $A_v$ in $A$ of a vertex $v$ 
on the arcs emanating from $v$. The orbit of $A_v$ containing the arcs $(v,w)$ 
and $(v,w')$ are paired if the arc $(v,w')$ lies in the same orbit of $A$ as the arc $(w,v)$.  

We define the arc-type of $X$ as the partition $\Pi$ of $d$ as the sum  
\begin{equation}\tag{$\dagger$} 
\label{arctype} 
\Pi \ = \ n_1 + n_2 + \dots + n_t + (m_1 + m_1) + (m_2 + m_2) + \dots + (m_s + m_s)
\end{equation} 
where $n_1, n_2, \dots, n_t$ are the sizes of the self-paired arc-orbits of $A_v$ 
on the arcs emanating from $v$, 
and  $m_1,m_1, m_2,m_2, \dots, m_s,m_s$ are the sizes of the non-self-paired 
arc-orbits, in descending order.  

Similarly, the {\em edge-type\/} of $X$ is the partition of $d$ as the sum of the sizes of the 
orbits of $A_v$ on edges incident with $v$, and can be found by simply replacing each 
bracketed term $(m_j+m_j)$ by $2m_j$ (for $1 \le j \le s$). 

The number of possibilities for the arc-type $\Pi$ depends on the valency $d$. 
For $d = 1$ there is just one possibility, namely with $n_1 = 1$, and this occurs for 
the complete graph $K_2$. 
For $d = 2$, in principle there could be three possibilities, namely $2$, $1+1$ and $(1+1)$, 
but every 2-valent connected graph is a cycle, and is therefore arc-transitive, with arc-type $2$.  
In particular, $1+1$ and $(1+1)$ cannot occur as arc-types. 
For $d = 3$ there are four possibilities (namely $3$, $2+1$, $1+1$ and $(1+1)$), 
and they all occur, as shown in \cite{Coxeter}. 
A natural question arises as to what arc-types occur for higher valencies. 

In this paper, we provide some basic theory for arc-types, which helps us to answer 
that question. 
In particular, we give for each positive integer $d$ the number of different 
partitions of the above form (\ref{arctype}) for $d$, by means of a generating function. 
This gives a closed form solution for the number of different possibilities 
in the case of a GRR of given valency $d$.  (As a curiosity, we mention there is 
also a connection with the different root types of polynomials with real coefficients.) 

Then as our main theorem, we show that with the exception of $1+1$ and $(1+1)$, 
every partition $\Pi$ as defined in (\ref{arctype}) is \emph{realisable}, 
in the sense that there exists at least graph with $\Pi$ as its arc-type.
To prove this, we consider how to combine `small' vertex-transitive graphs 
into a larger vertex-transitive graph, preserving (but increasing the number of) 
the summands in the arc-type.
The key step is to show that the arc-type of a Cartesian product 
of such graphs is just the sum of their arc-types, 
when the graphs are `relatively prime' with respect to the Cartesian product. 

Our proof of the theorem then reduces to finding suitable `building blocks', 
to use as base cases for the resulting construction.  
Several interesting families and examples of graphs are found to be helpful.  
In particular, we introduce the concept of a special kind of \emph{thickened cover} 
of a graph, obtained by replacing some edges of the given graph by 
complete bipartite graphs, and other edges by ladder graphs (with `parallel' edges).
Under some special conditions, the thickened cover is vertex-transitive, 
and it is easy to compute its arc-type from the arc-type of the given base graph.  

Finally, as a corollary of our main theorem, we show that every 
standard partition of a positive integer $d$ is realisable as the 
{\em edge-type\/} of a vertex-transitive graph  of valency $d$, 
except for $1+1$ (when $d = 2$). 

Vertex-transitve graphs are key players in algebraic graph theory, 
but also (as intimated earlier) they have important applications in 
other branches of mathematics.  In group theory they play a crucial role 
as Cayley graphs.  In geometry they are encountered in 
convex and abstract polytopes, incidence geometries, and configurations, 
and in manifold topology they feature in the study of regular and chiral maps 
and hypermaps, and Riemann and Klein surfaces with large automorphism groups.  

Classification of vertex-transitive graphs by their edge- or arc-type gives
a new viewpoint, and helps provide a better understanding of their structure. 
This approach can also be fruitful in terms of determining all small examples 
of various kinds of graphs, akin to the census of $3$-valent symmetric graphs 
on up to 10000 vertices \cite{Foster,ConderFosterCensus,Conder10000}, 
the census of vertex-transitive graphs up to 31 vertices \cite{Royle}, 
or the census of small $4$-valent half-arc-transitive graphs 
and arc-transitive digraphs of valency $2$ \cite{PSVhat}. 
For example, the construction used by Poto\v{c}nik, Spiga and Verret to 
obtain their census of vertex-transitive $3$-valent graphs on up to 1280 vertices in \cite{PSVvt}
depends on the edge-type, 
and their census of all connected quartic arc-transitive graphs of order up to $640$ 
(also in \cite{PSVvt}) was obtained by associating some of them 
with vertex-transitive $3$-valent graphs of edge-type $2+1$ (and using cycle decompositions); 
see also \cite{PSV4valent,PW}.
In these cases it was a stratified approach that enabled the limits of the census 
to be pushed so high, and it is likely that for graphs of higher valency or larger order, 
this kind of approach will be invaluable.   


\section{Preliminaries}
\label{sec:prelims}

All the graphs we consider in this paper are finite, simple, undirected and non-trivial 
(in the sense of containing at least one edge).  
Given a graph $X$, we denote by $V(X)$ and $E(X)$ the set of vertices
and the set of edges of $X$, respectively. We denote an edge of $X$ with vertices 
$u$ and $v$ by $\{u,v\}$, or sometimes more briefly by $uv$. 
For any vertex $v$ of $X$, we denote by $E(v)$ the set of edges of $X$ incident with $v$.

Associated with each edge $\{u,v\}$ there are two {\em ordered edges} 
(also called \emph{arcs\/}) of $X$, which we denote by $(u,v)$ and $(v,u)$,  
and we denote by $A(X)$ the set of arcs of $X$.  Also for a vertex $v$ of $X$, 
we define $A(v)=\{(v,u): (v,u) \in A(X) \}$ as the set of arcs of $X$ that emanate 
from (or start at)~$v$.

The automorphism group of $X$ is denoted by $\Aut(X)$.
Note that the action of  $\Aut(X)$ on the vertex-set $V(X)$ also induces an action
of $\Aut(X)$ on the edge-set $E(X)$ and one on the arc-set $A(X)$.
If the action of $\Aut(X)$ is transitive on the vertex-set, edge-set, or arc-set, 
then we say that $X$ is \emph{vertex-transitive}, \emph{edge-transitive} or \emph{arc-transitive}, 
respectively.  
An arc-transitive graph is often also called  \emph{symmetric}.
The graph $X$ is \emph{half-arc-transitive} if it is vertex-transitive and edge-transitive, 
but not arc-transitive.
Note that the valency of a half-arc-transitive graph is necessarily even;
see \cite[p. 59]{Tutte}.

In this paper we only consider vertex-transitive graphs. 
Obviously, vertex-transitive graphs are always regular. 
Moreover, because a disconnected vertex-transitive graph consists of pairwise isomorphic 
connected components, we may restrict our attention here 
to connected graphs.  Also we will sometimes use `VT' as an abbreviation for vertex-transitive. 

Next, let $G$ be a group, and let $S$ be  a subset of $G$ that is inverse-closed 
and does not contain the identity element. Then the  \emph{Cayley graph} $\Cay(G, S)$
is the graph with vertex-set $G$, and with vertices $u$ and $v$ 
being adjacent if and only if $vu^{-1} \in S$ 
(or equivalently, $v = xu$ for some $x \in S$). 
Since we require $S$ to be inverse-closed, this Cayley graph is undirected, 
and since $S$ does not contain the identity, the graph has no loops. 
Also $\Cay(G, S)$ is regular, with valency $|S|$, 
and is connected if and only if $S$ generates $G$.
Furthermore, it is easy to see that $G$ acts as a group of automorphisms 
of $\Cay(G, S)$ by right multiplication, and this action is vertex-transitive, with trivial stabiliser, 
and so this action of $G$ on $\Cay(G, S)$ is sharply-transitive (or {\em regular\/}). 
Hence in particular, $\Cay(G, S)$ is vertex-transitive.  

Indeed the following (which is attributed to Sabidussi  \cite{Sabidussi})  
shows how to recognise Cayley graphs: 

\begin{Lemma}                             \label{lemma:sabidussi}
A  graph $X$ is a Cayley graph for the group $G$ if and only if  
$X$ is vertex-transitive and $G$ acts regularly on $V(X)$.
More generally, a graph $X$ is a Cayley graph if and only if some subgroup of $\Aut(X)$
acts regularly on $V(X)$.
\end{Lemma}

Observe that for a fixed $x \in S$, all the  edges of form $\{u,xu\}$ and 
$\{u,x^{-1}u\}$ for $u \in G$ lie in the same edge-orbit (as each other) under the automorphism group 
of $\Cay(G, S)$, and similarly, that all the arcs of the form $(u,xu)$  lie in the same arc-orbit.
If all such arc-orbits are distinct, then $G$ is the full automorphism group of $\Cay(G, S)$, 
and $\Cay(G, S)$ is called a \emph{graphical regular representation} of the group $G$, 
or briefly a $GRR$ for $G$.  

Another term for such a graph is \emph{zero-symmetric}.  
Note that if $X$ is a connected zero-symmetric graph (or GRR) with valency $d$, 
then $X$ has $d$ arc-orbits, 
and all the arcs emanating from a given vertex $v$ lie in different arc-orbits. 
Moreover, if $G = \Aut(X)$, 
then the stabiliser $G_v$ of any vertex $v$ must fix each neighbour of $v$,  
and then by connectedness, it follows that $G_v$ is trivial. 
Thus $G$ acts regularly on $V(X)$, and so by Lemma \ref{lemma:sabidussi}, 
the graph $X$ is a Cayley graph for $G$. 

Next, we describe another group-theoretic construction, for a special class 
of vertex-transitive graphs.
Let $G$ be a group, let $H$ be a subgroup of $G$, and let $a$ be an element of $G$
such that $a^2 \in H$.  Now define a graph $\Gamma = \Gamma(G,H,a)$ by setting 
$$
V(\Gamma) = \{Hg : \, g \in G\} 
\quad \hbox{and} \quad 
E(\Gamma) = \{\{Hx,Hy\} : x,y \in G \mid xy^{-1} \in HaH\}.
$$
This graph $\Gamma$ is called a (\emph{Sabidussi}) \emph{double coset graph}.
As with Cayley graphs, the given group $G$ induces a group of automorphisms 
of $\Gamma(G,H,a)$ by right
multiplication, since $(xg)(yg)^{-1} = xgg^{-1}y^{-1} = xy^{-1} \in HaH$
whenever $\{Hx,Hy\} \in E(\Gamma)$.
Again this action is vertex-transitive, since $(Hx)x^{-1}y = Hy$.
Moreover, the stabiliser of the vertex $H$ is $G_H = \{g \in G | \, Hg =H\}$, 
which is $H$ itself, 
and this acts transitively on the neighbourhood $\{ Hah : h \in H\}$ of $H$, 
so in fact $\Gamma(G,H, a)$ is arc-transitive.
Conversely, every non-trivial arc-transitive graph $X$ can be constructed in this way, 
by taking $G = \Aut(X)$, and $H = G_v$ for some $v \in V(X)$, and $a$ as any 
automorphism in $G$ that interchanges $v$ with one of its neighbours. 

Finally, we mention a convenient way to describe cubic Hamiltonian graphs, 
that will be useful later.  Let $X$ be a cubic Hamiltonian graph on $n$ vertices.
Label the vertices of $X$ with numbers $0,\dots, n-1$, such that vertices $i$ and $i+1$ 
are consecutive in a given Hamilton cycle for $0 \le i < n$ (treated modulo $n$). 
Each vertex $i$ is adjacent to $i-1$ and $i+1$ (mod $n$), and to one other vertex, 
which has label $v_i$, say.  Now define $d_i = v_i-i$ for $0 \le i < n$.
Then the \emph{LCF-code} of  $X$ is given by the sequence 
$[d_0,\dots,d_{n-1}]$. 
Clearly the LCF-code defines the graph $X$, since the edges are $\{i,i+1\}$ for all $i$ (mod $n$) 
and $\{i,i+d_i\}$ for all $i$ (mod $n$).  On the other hand, The LCF is not necessarily unique for $X$ 
(since it depends on the choice of Hamilton cycle).  
Note also that if the code sequence is periodic, then it is sufficient to list a sub-sequence 
and indicate how many times it repeats, using a superscript. 
For example, $[3,-3]^4$ is the LCF-code of a 3-dimensional cube. 

\section{Edge-types and integer partitions}
\label{sec:edge-type} 

Let $X$ be a vertex-transitive graph of valency $d$, 
and let $\{\Delta_1,\dots,\Delta_k\}$ be the set of orbits of  $G = \Aut(X)$ on $E(X)$.
This partition of $E(X)$ into orbits also induces a partition of the set $E(v)$ of all 
edges incident with a given vertex $v$, namely into the sets $E(v) \cap \Delta_i$ for $1\le i \le k$. 
These are simply the restrictions of the edge-orbits $\Delta_i$ to the set $E(v)$. 

If we let $\ell_i = |E(v) \cap \Delta_i|$ for each $i$, then we may define the \emph{edge-type} 
of  $X$ to be the partition of the valency $d$ as the sum 
$$
\ell_1 + \ell_2 +\dots + \ell_k.
$$
where we assume that the numbers $\ell_i$ are in descending order. 
Note that by vertex-transitivity, the numbers $\ell_i$ do not depend 
on the choice of $v$. 
(Indeed when $X$ is finite, counting incident vertex-edge edge pairs $(v,e)$ with $e \in \Delta_i$ 
gives $2|\Delta_i| = |V(X)|\ell_i$ and therefore $\ell_i = 2|\Delta_i|/|V(X)|$ for each $i$.)  
Hence in particular, the edge-type does not depend on the choice of $v$. 
We denote the edge-type of $X$ by ${\rm et}(X)$. 

It is not at all obvious as to which partitions can occur as the edge-type of a vertex-transitive graph. 
The edge-type of a vertex-transitive cubic graph can be $3$, or $2+1$, 
or $1+1+1$, while that of a vertex-transitive quartic graph can be $4$, or $3+1$, or $2 + 2$, 
or $2 + 1 + 1$, or $1 + 1 + 1 + 1$.   
We will see instances of all of these in Section \ref{sec:small_examples}. 
Then later, in Section \ref{sec:realise},
we will show that with the exception of $1+1$ (for $d = 2$), 
every standard partition of a given positive integer $d$ can be realised as an edge-type. 

To enumerate the possibilities for a given valency $d$, we may
use generating functions for integer partitions. 
Let $p(d,k)$ denote the number of partitions of $d$ with $k$ parts, 
and let $p(d) $ denote the number of partitions of an integer $d$.
Obviously, $p(d) = \sum_k p(d,k)$.

The generating functions for integer partitions are very well-known, and can be found in 
\cite[p. 100]{Wilf} for example.  
In fact, the generating function $P(x,y)$ for $p(d,k)$ is given by 
$$
P(x,y) = \sum_d \sum_{k \,\ge \, 0} p(d,k)  y^kx^d
            = \prod_{n \,\ge\, 1} \frac{1}{1 - yx^n} \, ,
$$
and then by taking $y=1$, we get
$$
P(x)  = \prod_{n \,\ge\, 1} \frac{1}{1 - x^n}  = \sum_{d \,\ge\, 0} p(d) x^d 
$$
as the generating function for $p(d)$ itself.

\section{Arc-types and marked partitions}
\label{sec:arc-types}

In this section we refine the notion of edge-type of a vertex-transitive graph $X$, 
by considering the action of $\Aut(X)$ on the arcs of $X$. 

Let $\Delta$ be an orbit of  $\Aut(X)$ on $A(X)$, 
and let $\Delta^* = \{ (v,u) : (u,v) \in \Delta \}$ be its {\em paired} orbit, in the same 
way that a permutation group on a set $\Omega$ has paired orbitals on 
the Cartesian product $\Omega \times \Omega$. 
Note that $\Delta^*$ is also an orbit of $\Aut(X)$ on $A(X)$, 
and that the union $\Delta \cup \Delta^*$ consists of all the arcs obtainable 
from an orbit of $\Aut(X)$ on edges of $X$. 
We say that $\Delta$ is \emph{self-paired} if $\Delta = \Delta^*$ 
and \emph{non-self-paired} if $\Delta \not = \Delta^*$.

We can now write the orbits of $\Aut(X)$ on $A(X)$ as 
$$\Delta_1, \ldots, \Delta_{t}, \ \Delta_{t+1},\Delta_{t+1}^{*}, \dots, \Delta_{t+s},\Delta_{t+s}^{*},$$
where $\Delta_1, \ldots, \Delta_t$ are self-paired, 
while $\Delta_{t+1}, \ldots, \Delta_{t+s}$ are non-self-paired. 

This partition of $A(X)$ into these orbits of $\Aut(G)$ also induces a partition of the set of arcs 
emanating from a given vertex. For any vertex $v$ of $X$, define 
$n_i = |A(v) \cap \Delta_i|$ for $1 \le i \le t$, and $m_j = |A(v) \cap \Delta_j|$ for $t+1 \le j \le t+s$. 
Again since $X$ is vertex-transitive, these numbers do not depend on the choice of $v$.
In particular, it follows that $|A(v) \cap \Delta_j^*| = |A(v) \cap \Delta_j| = m_j$ for $t+1 \le j \le t+s$. 

Hence the $n_i$ are the sizes of the self-paired arc-orbits restricted to $A(v)$, 
while the $m_j$ are the sizes of the non-self-paired arc-orbits restricted to $A(v)$, 
and thus we obtain 
$$
 d \, = \, |A(v)| \, = \, n_1 +\dots + n_t + (m_1 + m_1) + \dots + (m_s + m_s), 
$$
just as in (\ref{arctype}) in the Introduction. 
This is the {\em arc-type\/} of $X$, and we denote it by ${\rm at}(X)$. 

We may call the expression on the right-hand-side of the above a {\em marked partition\/} 
of the integer $d$.   
By this, we mean simply a partition in which some pairs of equal-valued 
summands are placed in parentheses, or bracketed. 
Note that the parentheses in a marked partition $\Pi$ are important, 
because when $\Pi$ represents the arc-type of a VT graph, 
they indicate that the two numbers summed between the parentheses are the 
sizes of two paired arc-orbits corresponding to the same edge-orbit.

Since the order of the arc-orbits of each of the two kinds (self-paired and 
non-self-paired) can be chosen arbitrarily, we may consider two marked partitions 
to be equal if they have the same summands, possibly in a different order.  
Usually we will assume that the summands of each kind (unbracketed and bracketed) 
are in descending order, so that $n_1 \ge \dots \ge n_t$ and $m_1 \ge \dots \ge m_s$. 

\smallskip
In a sense, the three most important classes of vertex-transitive graphs 
are the arc-transitive, half-arc-transitive and zero-symmetric graphs, 
and their arc-types are as follows: 
\\[-18pt] 
\begin{itemize}
\item Arc-transitive graphs of valency $d$ have arc-type $d$; \\[-20pt] 
\item Half-arc-transitive graphs of even valency $d$ have arc-type  $(d/2 + d/2)$; \\[-20pt] 
\item Zero-symmetric graphs have arc-type
           $1 + \ldots + 1 + (1+1) + \ldots + (1+1)$. 
\end{itemize}
In particular, it follows that there are $\lfloor (d+1)/2 \rfloor$ possibilities for the 
arc-type of a $d$-valent zero-symmetric graph (or GRR). 

\smallskip
At this point, we note that the arc-type of a VT graph $X$ 
with automorphism group $G$ can be determined by considering solely the 
action of the vertex-stabiliser $G_v$ on the neighbourhood $X(v)$ of an given vertex $v$. 
This observation follows from standard theory of transitive permutation groups. 
If $\Delta$ is an arc-orbit of $G$, then the 
set $\Delta(v) = \{w : w \in V(X) \mid (v,w) \in \Delta\}$ is an orbit of $G_v$ on $X(v)$, 
and conversely, if $\Lambda$ is an orbit of $G_v$ on $X(v)$, 
then $\{(v,w)^g : u \in \Lambda, g\in G\}$ is a orbit of $G$ on $A(X)$. 
Also the arc-orbit $\Delta$ is self-paired if and only if $\Lambda$ is a
self-paired orbit of $G_v$ on $X(v)$, since both mean that the 
arc $(w,v)$ lies in the same orbit as $(v,w)$, for some $w \in \Lambda$. 
Hence in particular, the summands in the arc-type of $X$ are just the sizes 
of the orbits of $G_v$ on the neighbourhood of a vertex $v$ of $X$.

\smallskip
Next, by finding the generating function for the set of marked partitions, we can
also count the (maximum) number of possible arc-types for each valency $d$.

Let $t^\prime(d,k)$ denote the number of marked partitions of $d$ with $k$ parts,  
and let $t(d) $ denote the number of marked partitions of an integer $d$.
Obviously, $t(d) = \sum_k t^\prime(d,k)$.
We can obtain the generating function $T^\prime(x,y)$ for $t^\prime(d,k)$ 
by adapting the generating function for standard partitions, to take account of 
the bracketed pairs.  This can be found from \cite[p. 95]{Wilf}, for example, 
and is as follows:
$$
T^\prime(x,y)  = \sum_{d \ge 0} \sum_{k \ge 0} t^\prime(d,k)  x^dy^k
            = \prod_{n \ge 1} \frac{1}{(1 - yx^n) (1 - y^2 x^{2n})}.
$$
Then by taking $y=1$, we get
$$
T(x) =   T^\prime(x,1) = \prod_{n \ge 1} \frac{1}{(1 - x^n) (1 - x^{2n})} =  
               \sum_{d \ge 0} t(d)  x^d
$$
as the generating function for $t(d)$ itself.

Here we remark that  a different combinatorial approach can be taken 
for the generating function $T(x)$, namely through refining integer partitions by
labelling some even parts. 
(For example, the partition $6=2+2+1+1$ gives rise to three labelled partitions: 
$6=2+2+1+1$, $6=2+2^*+1+1$, and $6=2^*+2^*+1+1$.) 
Now let $t^*(d,k)$ denote the number of labelled partitions of an integer $d$ 
having $k$ parts. Then the generating function for $t^*(d,k)$ is 
$$
T^*(x,y)  = \sum_{d \ge 0} \sum_{k \ge 0} t^*(d,k)  x^dy^k
            = \prod_{n \ge 1} \frac{1}{(1 - yx^n) (1 - y x^{2n})},
$$
and we find that $T^*(x,1)=T^\prime(x,1)=T(x)$.\\

Also we note that the generating function $T(x)$ defines the  sequence 
$$1,1,3,4,9,12,23,31,54,73,118, \ldots ,$$ which is denoted by A002513 in The On-Line
Encyclopedia of Integer Sequences \cite{Sloane}.

Finally, it should come as no surprise that marked partitions of a positive integer $d$ 
can be used also to count different types of solutions of a real polynomial equation 
of degree $d$, when attention is paid whether the roots are real and unequal, 
real and equal (in various combinations)  or simple or multiple complex conjugate;  
see \cite{Capobianco}.

\section{Edge-types and arc-types  for small valency}
\label{sec:small_examples}

In this section we give examples of vertex-transitive graphs with  
every possible edge-type and arc-type, for valencies up to 4, 
and summarise the information in Table \ref{table:at} at the end. 
Note that it is enough to find examples of VT graphs for each arc-type, since the
same graphs will give also all the possible edge-types.
We do not give a proof that the arc-type is as claimed in each case, since that 
can be easily verified by computer (using for example Magma \cite{Magma}), 
or in some cases by hand.

Graphs with arc-type $1+1+1$ of order up to 120 are given in
\cite[Part III]{Coxeter}, and graphs with arc-type $(1+1)+1$ of order up to 120 
are given in \cite[Part II]{Coxeter}. 
The other zero-symmetric graphs listed here were found with the help 
of Magma \cite{Magma}, by checking the Cayley graphs for certain kinds of 
generating sets for small groups.  In some cases, we give the smallest 
possible example with the given arc-type. 
Some examples were found by also checking
tables of vertex-transitive graphs on up to 31 vertices; see  \cite{Royle}. 

\paragraph{Valency $d=1$ (one case):}

\begin{itemize}

\item[({\rm P1})]There is only one marked partition of $1$, namely $1$, and 
only one VT graph with this arc-type, namely  the complete graph $K_2$.
\end{itemize}

\paragraph{Valency $d=2$ (three cases):}
\begin{itemize}
\item[({\rm P2})] $2 = 2$: For every $n \ge 3$, the simple cycle $C_n$ has arc-type $2$.

\item[({\rm P3})] $2 = 1 + 1$: No VT graph has arc-type $1+1$, 
   because cycles are the only connected regular graphs with valency 2.

\item[({\rm P4})] $2 =  (1+1)$: No VT graph has arc-type $(1+1)$, 
    because cycles are the only connected regular graphs with valency 2.
    
\end{itemize}

\paragraph{Valency $d=3$ (four cases):}
%
\begin{itemize}

\item[({\rm P5})] $3 = 3$: The VT graphs with arc-type $3$ are precisely 
the arc-transitive cubic graphs, and there are infinitely many examples, 
the smallest of which is the complete graph $K_4$. 
Numerous other small examples, including the ubiquitous Petersen graph, 
are given in the Foster census \cite{Foster}, which was later expanded 
by Conder and Dobcs\'{a}nyi \cite{ConderFosterCensus}, 
and again further by Conder \cite{Conder2048} up to order $2048$.

\item[({\rm P6})] $3 = 2 + 1$: 
The smallest VT graph with arc-type $2+1$ is the triangular prism, on 6 vertices; 
see Figure \ref{fig:2_1}. It is easy to see that the edges on the two
triangles form one edge-orbit, while all the other edges form another orbit.
\begin{figure}[H]
\centering
\includegraphics[width= 3cm]{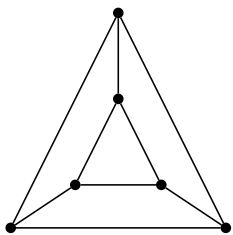} 
\caption{The triangular prism, which has arc-type $2+1$ }
\label{fig:2_1}
\end{figure}

\item[({\rm P7})] $3 = 1+1+1$:
The smallest VT graph with arc-type $1+1+1$ is the zero-symmetric
graph on 18 vertices from \cite[p. 4]{Coxeter}; see also Figure \ref{fig:1_1_1}.  
This is the Cayley graph of a group of order $18$ generated by three involutions,
and it can also be described as a cubic Hamiltonian graph with LCF-code $[5,-5]^9$.
\begin{figure}[H]
\centering
\includegraphics[width=3cm]{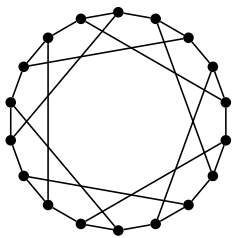}
\caption{The smallest VT graph with arc-type $1+1+1$, on $18$ vertices}
\label{fig:1_1_1}
\end{figure}  

\item[({\rm P8})] $3 =1+(1+1)$: 
The smallest VT graph with arc-type $1+(1+1)$ is the zero-symmetric
 graph on 20 vertices from \cite[p. 35]{Coxeter}; see  Figure \ref{fig:zerosym20a}.
This is the Cayley graph of a group of order $20$, generated by one involution and one non-involution, 
and it can also be described as a cubic Hamiltonian graph with LCF-code $[6,6,-6,-6]^5$. 
For more properties of this graph, see Lemma  \ref{lemma:typem(11)}.
\begin{figure}[H]
\centering
\includegraphics[width= 4cm]{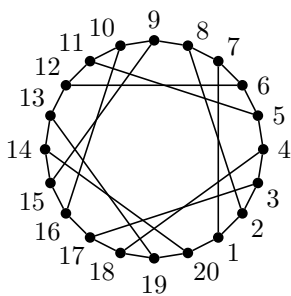}
\caption{The smallest VT graph with arc-type $1+(1+1)$, on $20$ vertices}
\label{fig:zerosym20a}
\end{figure}
\end{itemize}

\paragraph{Valency $d=4$ (nine cases):}

\begin{itemize}

\item[({\rm P9})] $4 = 4$:  
The VT graphs with arc-type $4$ are arc-transitive quartic graphs. 
The smallest example is the complete graph $K_5$.
                     
\item[({\rm P10})] $4 = (2+2)$: 
The VT graphs with arc-type $(2+2)$ are half-arc-transitive quartic graphs.
The smallest example is the Holt graph \cite{Holt}, of order 27 see Figure \ref{fig:22}.
\begin{figure}[H]
\centering
\includegraphics[width= 5cm]{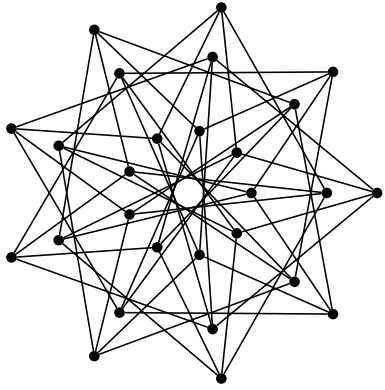}
\caption{The Holt graph (the smallest $4$-valent half-arc-transitive graph)}
\label{fig:22}
\end{figure}

\item[({\rm P11})] $4 = 3+1$: 
The smallest VT graph with arc-type $3+1$ is $K_4 \cp K_2$, 
which is the Cartesian product of $K_4$ and $K_2$.  
The two summands $K_4$ and $K_2$ have arc-types 
$3$ and $1$, respectively, and are `relatively prime'.  
The example will be generalised in Theorem  \ref{thm:Cartesian}.
                     
\item[({\rm P12})] $4 = 2 + 2$: 
The smallest VT graph with arc-type $2+2$ is the circulant graph
$\Cay(\bZ_7;\{1,2\})$ on 7 vertices, where $\bZ_7$ is viewed as an 
additive group.  This graph is shown in Figure \ref{fig:1_11}.
Each of the edges on the outer 7-cycle lies in two triangles while
each edge of the inner 7-cycle lies in only one triangle, and it follows that 
$\Cay(\bZ_7;\{1,2\})$ has two edge orbits.
\begin{figure}[H]
\centering
\includegraphics[width= 3cm]{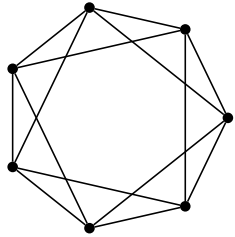}
\caption{The circulant $\Cay(\bZ_7;\{1,2\})$, which has arc-type $2+2$}
\label{fig:1_11}
\end{figure} 

\item[({\rm P13})] $4 = 2 + (1 + 1)$: 
The  graph on 40 vertices in Figure \ref{fig:211} 
is the smallest known VT graph with arc-type $2+(1+1)$. 
It is a thickened cover of the trivalent graph on 20 vertices
with arc-type $(1+1) +1$; see Section \ref{sec:thcov} for generalisations of this.
\begin{figure}[H]
\centering
\includegraphics[width = 4cm]{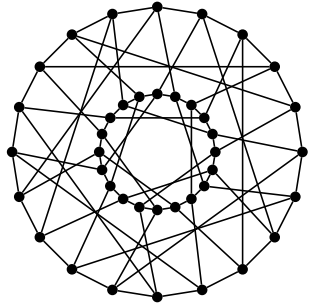}
\caption{A VT graph with arc-type $2 + (1 + 1)$, on $40$ vertices}
\label{fig:211}
\end{figure}

\item[({\rm P14})] $4 = (1+1)+(1 + 1)$:  
The  graph on 42 vertices in Figure \ref{fig:11_11} 
is the smallest VT graph with arc-type $ (1+1)+(1+1)$.
As a GRR, it is a Cayley graph of the group $C_7 \rtimes C_6$ with  generating set
that contains an element of order 6, an element of order 7, and their inverses.
For more details on this graph see Lemma  \ref{lemma:type(11)(11)}.
\begin{figure}[H]
\centering
\includegraphics[width = 5cm]{fig7_new.tikz}
\hspace*{1cm}
\includegraphics[width = 5cm]{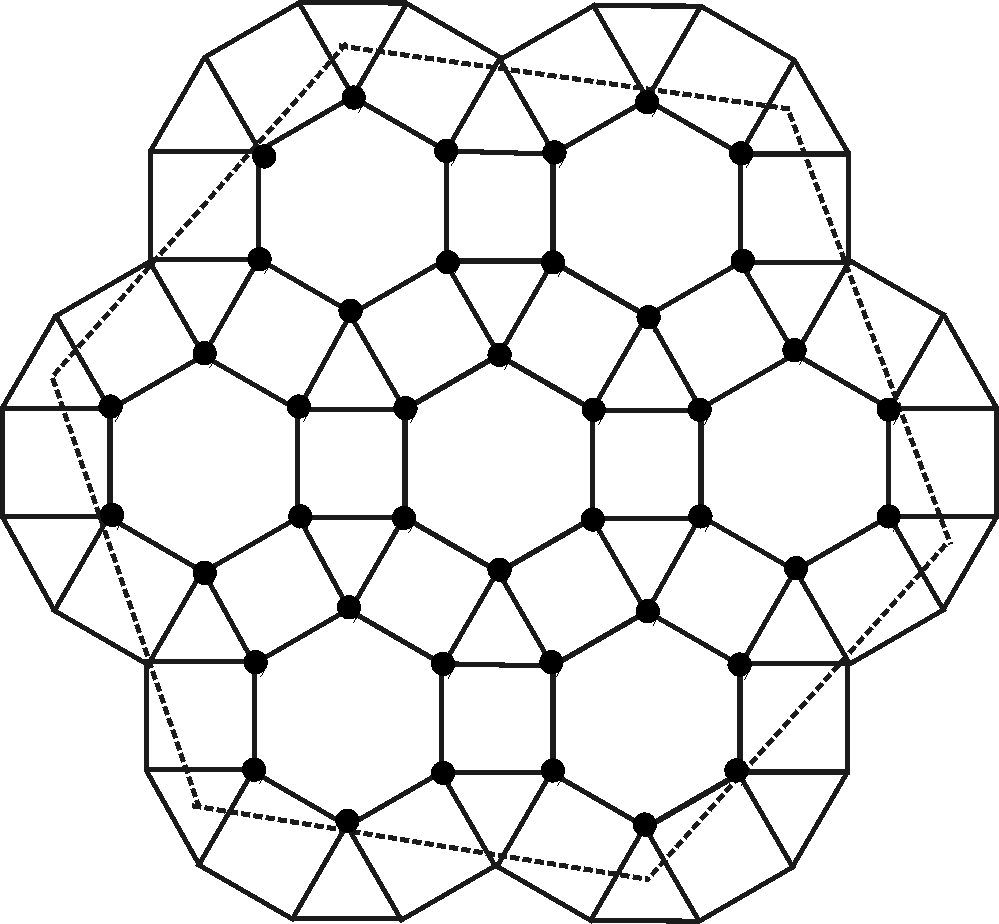}
\caption{On the left is the smallest VT graph with arc-type $(1+1)+(1+1)$, on $42$ vertices 
--- and on the right is an illustration of an embedding of this graph on the torus, 
using a hexagon with opposite sides identified}
\label{fig:11_11}
\end{figure}


\item[({\rm P15})] $4 = 2 + 1 + 1$: 
The  graph on 12 vertices in Figure  \ref{fig:2_1_1} is the smallest VT graph with arc-type $2+1+1$. 
It can be obtained from the hexagonal prism by adding diagonals to three non-adjacent quadrangles.
\begin{figure}[H]
\centering
\includegraphics[width= 3cm]{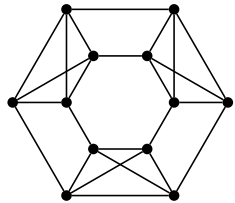}
\caption{The smallest VT graph with arc-type $2 + 1 + 1$, on $12$ vertices}
\label{fig:2_1_1}
\end{figure}

\item[({\rm P16})] $4 = 1 + 1 + (1+1)$: 
The graph  on 20 vertices in Figure \ref{fig:11_1_1} is the smallest VT graph 
with arc-type $1 + 1 + (1+1)$. 
If $G$ is the group Frobenius group $C_5 \rtimes C_4$ of order $20$, 
generated by the permutations $a = (1,2,3,4,5)$ and $b = (2,3,5,4)$, 
which satisfy the relations $a^5 = b^4 = 1$ and $b^{-1}ab = a^2$, 
then this graph is the Cayley graph (in fact a GRR) for $G$ given by the generating 
set $S = \{ab^2, a^2b^2, b, b^{-1}\}$. 
\begin{figure}[H]
\centering
\includegraphics[width = 4cm]{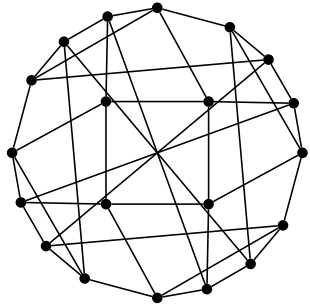}
\caption{The smallest VT graph with arc-type $1 + 1 + (1+1)$, on $20$ vertices}
\label{fig:11_1_1}
\end{figure}

\item[({\rm P17})] $4 = 1+1+1+ 1$: 
The graph  on 16 vertices in Figure \ref{fig:1111} is the smallest VT graph 
with arc-type $1+1+1+1$. 
It is the Cayley graph (in fact a GRR) for the dihedral group 
$D_8 = \langle\, x,y \ | \ x^2 = y^8 = (xy)^2 = 1 \, \rangle$ of order $16$ 
with generating set $S = \{x,xy,xy^2,xy^4 \}$. 
\begin{figure}[H]
\centering
\includegraphics[width = 4cm]{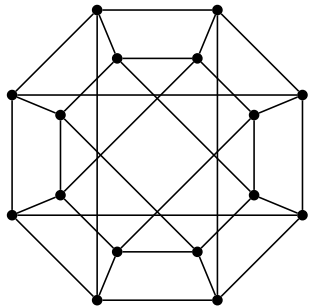}
\caption{The smallest VT graph with arc-type $1+1+1+1$, on $16$ vertices}
\label{fig:1111}
\end{figure}

\end{itemize}

The above examples are summarised in Table \ref{table:at}.

\begin{table}[H]
\begin{center}
\begin{tabular}{|l|l|l|l|l|}
\hline
Valency & Edge-type & Arc-type & Example & Case \\
\hline
$1$ & $1$ & $1$ & $K_2$ & P1 \\ \hline

$2$ & $2$ & $2$ & $C_3$ & P2 \\  \cline{3-5}
      & & $(1+1)$ & [Impossible] & P3 \\ \hline
      
$2$ & $1+1$ & $1+1$ & [Impossible] & P4 \\  \hline

$3$ & $3$ & $3$ & $K_4$ & P5 \\  \hline

$3$ & $2+1$ & $2+1$ & prisms & P6 \\  \cline{3-5}
      & & $(1+1)+1$ & $LCF[6,6,-6,-6]^5$ & P7 \\ \hline

$3$ & $1+1+1$ & $1+1+1$ & $LCF[5,-5]^9$ & P8 \\ \hline

$4$ & $4$ & $4$ & $K_5$ & P9 \\   \cline{3-5}
      & & $(2+2)$ & Holt graph & P10 \\ \hline

$4$ & $3+1$ & $3+1$ & $K_4 \cp K_2$ & P11 \\ \hline

$4$ & $2+2$ & $2+2$ & $\Cay(\bZ_7;\{1,2\})$ & P12 \\   \cline{3-5}
      & & $2+(1+1)$ & Figure \ref{fig:211} & P13 \\   \cline{3-5}
      & & $(1+1)+(1+1)$ & Figure \ref{fig:11_11}  & P14    \\ \hline

$4$ & $2+1+1$ & $2+1+1$ & Figure \ref{fig:2_1_1} & P15 \\   \cline{3-5}
      & & $(1+1)+1+1$ & Figure \ref{fig:11_1_1} & P16 \\ \hline
      
$4$ & $1+1+1+1$ & $1+1+1+1$ & Figure \ref{fig:1111} & P17 \\ 
\hline
\end{tabular}     
\caption{Edge-types and arc-types of VT graphs with valency up to 4}
\label{table:at}
\end{center}
\end{table}

\section{Arc-types of Cartesian products}
\label{sec:products}

Given a pair of graphs $X$ and $Y$ (which might or might not be distinct), 
the Cartesian product $X \cp Y$  is a graph 
with vertex set $V(X) \times V(Y)$, such that two vertices $(u,x)$ and $(v,y)$
are adjacent in $X \cp Y$ if and only if 
$u = v$ and $x$ is adjacent with $y$ in $Y,$ or
$x=y$ and $u$ is adjacent with $v$ in $X$. 

This definition can be extended to the Cartesian product $X_1 \cp \dots \cp  X_k$ 
of a larger number of graphs $X_1,\dots,X_k$.  
The terms $X_i$ are called the {\em factors\/} of the Cartesian product $X_1 \cp \dots \cp  X_k$. 
The Cartesian product operation $\cp$ is associative and commutative. 
A good reference for studying this and other products is the book by 
Imrich and Klav\v{z}ar \cite{ImrichKlavzar}.

There are many properties of Cartesian product graphs that can
be easily derived from the properties of their factors. 
For example, we have the following: 

\begin{Proposition}                                   \label{prop:connected}
A Cartesian product graph is connected if and only if all of its factors are connected.
\end{Proposition}

\begin{Proposition}                                   \label{prop:valency}
Let $X_1, \dots, X_k$ be regular graphs with valencies $d_1,\dots,d_k$.
Then their Cartesian product $X_1 \cp \dots \cp  X_k$ is also regular, with valency $d_1+\dots+d_k$.
\end{Proposition}

A graph $X$ is called \emph{prime} (with respect to the Cartesian product)
if it is not isomorphic to a Cartesian product of a pair of smaller, non-trivial graphs. 
It is well-known that every connected graph 
can be decomposed to a Cartesian product of prime graphs, which is unique 
up to reordering and isomorphism of the factors; for a proof, see \cite[Theorem 4.9]{ImrichKlavzar}.
Similarly, two graphs are said to be \emph{relatively prime} (with respect to the Cartesian product)
if there is no non-trivial graph that is a factor of both. 
Note that two prime graphs are relatively prime unless they are isomorphic. 

We are interested in the question how the symmetries of individual graphs are 
involved in the symmetries of their product. Let  $X=X_1 \cp \dots \cp X_k$, 
and let $\alpha$ be an automorphism of one of the $X_i$.
Then $\alpha$ induces an automorphism $\beta$ of $X$, given by 
$$
\beta: (v_1,\dots,v_k) \mapsto (v_1, \dots, v_{i-1},v_i^{\,\alpha},v_{i+1},\dots,v_k). 
$$
The set of all automorphisms of $X$ induced in this way forms a subgroup 
of $\Aut(X)$, and if some of the factors of $X$ are isomorphic, then $\Aut(X)$ 
contains also other automorphisms that permute these factors among themselves, 
but if the factors of $X$ are relatively prime, then there are no other automorphisms.
Indeed we have the following:

\begin{Theorem} {\rm (\cite[Corollary 4.17]{ImrichKlavzar}) }    
\label{thm:autoproduct}
Let $X$ be the Cartesian product $X = X_1 \cp \dots \cp X_k$ of
connected relatively prime graphs $X_1, \dots, X_k$ .  Then every automorphism $\varphi$ of $X$ 
has the property that 
$$
\varphi: (v_1,\dots,v_k) \mapsto (v_1^{\,\varphi_1},\dots,v_k^{\,\varphi_k}) 
\quad \hbox{for all } (v_1,\dots,v_k) \in V(X),
$$
where $\varphi_i$ is an automorphism of $X_i$ for $1 \le i \le k$.
\end{Theorem}

\begin{Corollary}                                         
\label{cor:autoproduct}
If $X$ be the Cartesian product $X_1 \cp \dots \cp X_k$ of
connected relatively prime graphs $X_1, \dots, X_k$, then
$
\Aut(X) \cong \Aut(X_1) \times \dots \times \Aut(X_k).
$
\end{Corollary}


\begin{Corollary}                                               \label{cor:producttransitive}
A Cartesian product of connected graphs is vertex-transitive 
if and only if every factor is vertex-transitive. 
\end{Corollary}
%
Corollary~\ref{cor:producttransitive} follows directly from Theorem \ref{thm:autoproduct} 
when the factors are pairwise relatively prime.  But it is also true in the general case 
--- for a proof, see \cite[Proposition 4.18]{ImrichKlavzar}.

We now come to the key observation  we need to prove our main theorem. 

\begin{Theorem}                                           \label{thm:Cartesian}
Let $X_1, \dots, X_k$ be non-trivial connected vertex-transitive graphs, 
with arc-types $\tau_1, \dots,\tau_k$.
Then also $X =X_1 \cp \dots \cp X_k$  is a connected vertex-transitive graph,
and if $X_1, \dots, X_k$ are pairwise relatively prime, 
then the arc-type of $X$ is $\tau_1 + \dots + \tau_k$.
\end{Theorem}
\begin{proof}
First, the graph $X$ is connected by Proposition \ref{prop:connected}, 
and vertex-transitive by Corollary  \ref{cor:producttransitive}.
For the second part, suppose that $X_1, \dots, X_k$ are pairwise relatively prime.  
Then by Corollary \ref{cor:autoproduct}, 
we know that $\Aut(X) \cong \Aut(X_1) \times \dots \times \Aut(X_k)$. 
Moreover, by Theorem \ref{thm:autoproduct}, the stabiliser in $\Aut(X)$ of 
a vertex $(v_1,\dots,v_k)$ of $X$ is isomorphic to $\Aut(X_1)_{v_1} \times \dots \times \Aut(X_k)_{v_k}$.

We will now show that two arcs incident with a given vertex $u = (u_1,\dots,u_k)$ in $X$ 
are in the same orbit of $\Aut(X)$ if and only if the corresponding arcs are in the same 
orbit of $\Aut(X_i)$ for some $i$, and that two such arcs are in paired orbits of $\Aut(X)$ 
if and only if the corresponding arcs of $X_i$ belong to paired orbits of $\Aut(X_i)$ for some $i$.
This will imply that the sizes of arc-orbits of $\Aut(X)$ 
on $X$ match the sizes of arc-orbits of the subgroups $\Aut(X_i)$ 
on the corresponding $X_i$, 
for $1 \le i \le k$,  
and hence that the arc-type of $X$ is just the sum of the arc-types of $X_1,\dots,X_k$.

So suppose  $u' = (u_1^{\,\prime},\dots,u_k^{\,\prime})$ 
and $u'' = (u_1^{\,\prime\prime},\dots,u_k^{\,\prime\prime})$ 
are adjacent to $u = (u_1,\dots,u_k)$ in $X$, 
and that the arcs $(u,u')$ and $(u,u'')$ lie in the same orbit of $\Aut(X)$.
Then there exists an automorphism $\varphi$ of $X$ taking $(u,u')$ to $(u,u'')$, 
and since $\varphi$ stabilises $u$, we know that $\varphi=(\varphi_1,\dots,\varphi_k)$ 
where $\varphi_i \in \Aut(X_i)_{u_i}$ for $1 \le i \le k$.  
Also $u'$ differs from $u$ in only one coordinate, say the $i$-th one, 
in which case $u_j^{\,\prime} = u_j$ for $j \ne i$, 
and then since $\varphi=(\varphi_1,\dots,\varphi_k)$ takes $u'$ to $u''$ 
(and $\varphi_j$ fixes $u_j$), we find that $u_j^{\,\prime\prime} = u_j$, 
so that $u''$ differs from $u$ only in the $i$-th coordinate as well. 
In particular, $\varphi_i$ fixes $u_i$ and takes $u_i^{\,\prime}$ to $u_i^{\,\prime\prime}$, 
so the arcs $(u_i,u_i^{\,\prime})$ and $(u_i,u_i^{\,\prime\prime})$ lie in the same 
orbit of $\Aut(X_i)$.  

The converse is easy: if the arcs $(u_i,u_i^{\,\prime})$ and $(u_i,u_i^{\,\prime\prime})$ 
lie in the same orbit of $\Aut(X_i)$, and $\varphi_i$ is an automorphism of $X_i$ 
taking $(u_i,u_i^{\,\prime})$ to $(u_i,u_i^{\,\prime\prime})$, 
then letting $u_j^{\,\prime} = u_j^{\,\prime\prime} = u_j$  for $j \ne i$, 
and  $u' = (u_1^{\,\prime},\dots,u_k^{\,\prime})$ and $u'' = (u_1^{\,\prime\prime},\dots,u_k^{\,\prime\prime})$, 
we find that the automorphism of $X$ induced by $\varphi_i$ takes $(u,u')$ to $(u,u'')$, 
and so these two arcs lie in the same orbit of $\Aut(X)$.

On the other hand, suppose the arcs $(u,u')$ and $(u,u'')$ lie in different but paired orbits 
of $\Aut(X)$. Then there exists an automorphism $\varphi$ of $X$ taking $(u,u')$ to $(u'',u)$, 
and $\varphi=(\varphi_1,\dots,\varphi_k)$ where $\varphi_i \in \Aut(X_i)$ for $1 \le i \le k$.  
Again, $u'$ differs from $u$ in only one coordinate, say the $i$-th one, 
in which case $u_j^{\,\prime} = u_j$ for $j \ne i$, 
and since $\varphi$ takes $u'$ to $u$, we find that $\varphi_j$ fixes $u_j$, 
and then since $\varphi$ takes $u$ to $u''$, also $u_j^{\,\prime\prime} = u_j$. 
Thus, as before, $u''$ differs from $u$ only in the $i$-th coordinate, 
and the arcs $(u_i,u_i^{\,\prime})$ and $(u_i,u_i^{\,\prime\prime})$ lie in the same 
orbit of $\Aut(X_i)$.  
The converse is analogous to the previous case, and this completes the proof.
\end{proof}

Next, the following observations are often helpful when proving that a given 
graph is prime with respect to the Cartesian product.
The first two are easy, and the third one is proved in \cite{ImrichPeterin}, for example. 

\begin{Lemma}                                     \label{lem:notCartesian}
Let $X=X_1 \cp X_2$ be a Cartesian product of two connected
graphs, each of which has at least two vertices.
Then every edge of $X$ is contained in a $4$-cycle.
\end{Lemma}

\begin{Corollary}                                     \label{cor:notCartesian}
Let $X$ be a connected graph.
If some edge of $X$ is not contained in any $4$-cycle, then $X$ is prime.
In particular, if $X$ has no $4$-cycles, then $X$ is prime.
\end{Corollary}

\begin{Lemma}                                      \label{lem:notCartesian1}
Let $X$ be a Cartesian product of connected graphs. \\[-18pt] 
\begin{itemize}
\item[{\rm (a)}] All the edges in a cycle of length $3$ in $X$ belong to the same factor of $X$. \\[-20pt] 
\item[{\rm (b)}] Let $(v,u,w,x)$ be any $4$-cycle in $X$. Then the edges $\{v,u\}$ 
and $\{w,x\}$ belong to the same factor of $X$, as do the edges $\{u,w\}$ and $\{x,v\}$. \\[-20pt] 
\item[{\rm (c)}]  If $e$ and $f$ are incident edges that are not in the same factor of $X$,
then there exists a unique $4$-cycle that contains $e$ and $f$, and this $4$-cycle has no diagonals.
\end{itemize}
\end{Lemma}

\section{Thickened covers}
\label{sec:thcov}

In this section we explain the general notion of a thickened cover of a graph, 
and show how it can be used to build larger vertex-transitive graphs from a given one. 

\smallskip
Let $X$ be any simple graph, $F$ any subset of the edge-set of $X$, 
and $m$ any positive integer.  Then we define $X(F,m)$ to be the graph 
with vertex set $V(X) \times \bZ_m$, and with edges of two types: \\[-20pt] 
\begin{itemize}
\item[(a)] an edge from $(u,i)$ to $(v,i)$, for every $i \in \bZ_{m}$ 
  and every $\{u,v\}  \in E(X) \setminus F$, \\[-20pt]
\item[(b)] an edge from $(u,i)$ to $(v,j)$, for every $(i,j) \in \bZ_{m} \times \bZ_{m}$ 
  and every $\{u,v\} \in F$. \\[-15pt]
\end{itemize}
We call $X(F,m)$ a \emph{thickened m-cover} of $X$ over $F$.

\smallskip
In other words, a thickened $m$-cover of a graph $X$ over a given set $F$ of edges of $X$ 
is obtained by replacing each vertex of $X$  by $m$ vertices, and each edge by 
the complete bipartite graph $K_{m,m}$ if the edge lies in $F$, 
or by $mK_2$ (a set of $m$ `parallel' edges) if the edge does not lie in $F$. 

For example, the thickened $3$-cover of the path graph $P_6$ on $6$ vertices 
over the unique $1$-factor of $P_6$ is shown in Figure \ref{fig:thcover}.

\begin{figure}[H]
\centering
\includegraphics[width=6cm]{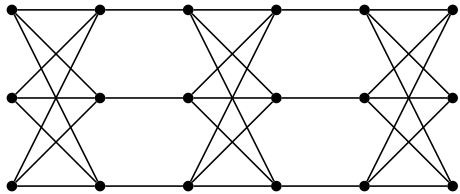}
\caption{A thickened $3$-cover of $P_6$ (over its $1$-factor)}
\label{fig:thcover}
\end{figure}   

Here we note that $X(F,1)$ is isomorphic to $X$, 
while $X(\emptyset,m)$ is isomorphic to $mX$ (the union of $m$ copies of $X$).  

Also the base graph $X$ is a quotient of $X(F,m)$, obtainable by identifying all 
vertices $(u,i)$ that have the same first coordinate, 
but $X(F,m)$ is not a covering graph of $X$ in the usual sense of that term 
when $F$ is non-empty and $m > 1$, because in that case the valency of a vertex $(u,i)$ 
of $X(F,m)$ is greater than the valency of $u$. 

On the other hand, if $X^{(m)}$ is the multigraph obtained from $X$ by replacing
each edge from $F$ by $m$ parallel edges, then $X(F,m)$ is a regular covering
graph of $X^{(m)}$, with voltages taken from $\bZ_m$: we may choose the voltages
of the edges not in $F$ to be $0$, and for each set of $m$ parallel edges, choose a 
direction and then assign distinct voltages from $\bZ_m$ to these edges. 
The derived graph of $X^{(m)}$ with this voltage assignment is a covering graph 
of $X^{(m)}$ that is isomorphic to $X(F,m)$.

(For the definitions of voltage graphs and covering graphs, see the book on 
topological graph theory by Gross and Tucker \cite{GrT}.) 

\smallskip 
%
For each $u \in V(X)$, we may call the set $\{(u,i) : \, i \in \bZ_m\}$ of vertices 
of $X(F,m)$ the \emph{fibre} over the vertex $u$ of $X$. 
Similarly the \emph{fibre} over the edge $\{u,v\}$ of $X$ is the 
set $\{\{(u,i),(v,i)\} : \, i \in \bZ_m\}$ of edges of $X(F,m)$ if $\{u,v\} \in E(X)\setminus F$, 
or the set $\{\{(u,i),(v,j)\} : \, i,j \in \bZ_m\}$ if $\{u,v\} \in F$. 
Also for each $i \in \bZ_m$ we call  the subgraph of $X(F,m)$ induced by the vertices 
$\{(u,i) : \, u \in V(X)\}$  the \emph{$i$-th layer} of $X(F,m)$.

We now define three families of bijections on the vertex set of $X(F,m)$.
The first is $\tilde{\varphi}$, which is induced by addition of $1$ mod $m$ on $\bZ_m$, 
and given by the rule
\begin{align}                                         \label{eq:varphi}
\tilde{\varphi}  &: (u,i) \mapsto (u,i+1) \quad \mbox{for all} \ u \in V(X) \ \mbox{and all} \ i \in \bZ_m.
\end{align}
Next, if $\psi$ is any automorphism of $X$, then we define $\tilde{\psi}$  by  the rule
\begin{align}                                            \label{eq:psi}
\tilde{\psi}  &: (u,i) \mapsto (u^\psi,i) \quad  \mbox{for all} \ u \in V(X) \ \mbox{and all} \ i \in \bZ_m.
\end{align}
Finally, if $i,j \in \bZ_m$ and $\{u,v\}$ is an edge in $F$ such that 
$u$ and $v$ lie in different components of $X \setminus F$ 
(the graph obtained from $X$ by deleting all the edges in $F$), 
then we define the bijection $\tilde\theta = \tilde{\theta}(u,v,i,j)$ by 
\begin{equation}                                                   \label{eq:theta}
\tilde{\theta} : \, (w,k) \mapsto \, \begin{cases}
                          (w,j) & \mbox{if} \ \ k=i  \ \mbox{and} \ w \
                          \mbox{lies in the same component of} \ X \setminus F \ \mbox{as} \ v,\\
                         (w,i) & \mbox{if} \ \ k=j  \ \mbox{and} \ w \
                          \mbox{lies in the same component of} \ X \setminus F \ \mbox{as} \ v,\\
                          (w,k)  & \mbox{otherwise.} \\
                           \end{cases}
\end{equation}

\begin{Lemma}                                                                   \label{lemma:varphi}
If $X$ is any  graph, $F \subseteq E(X)$ and $m \ge 2$, 
then $\tilde{\varphi}$ is an automorphism of $X(F,m)$.
\end{Lemma}
\begin{proof}
The mapping $\tilde{\varphi}$ is a bijection and obviously sends edges of
$X(F,m)$ to edges, so is an automorphism of $X(F,m)$.
\end{proof}

\begin{Lemma}                                                                 \label{lemma:psi}
If $X$ is a vertex-transitive graph, and $m \ge 2$, 
and $F \subseteq E(X)$ is a union of some edge-orbits of $X$, 
then  $\tilde{\psi}$  is an automorphism of $X(F,m)$ for every $\psi \in \Aut(X)$.
\end{Lemma}
\begin{proof}
The given mapping $\tilde{\psi}$ is clearly a bijection. 
Next, if $\{u,v\} \in F$ then also $\{u^\psi,v^\psi\} \in F$  by hypothesis, 
and so $\{(u,i),(v,j)\}^{\tilde{\psi}} = \{(\psi(u),i),(\psi(v),j)\}$ is an edge of $X(F,m)$, 
for all $i,j \in \bZ_{m}$.  
Similarly, if $\{u,v\} \in E(X) \setminus F$, then $\{u^\psi,v^\psi\} \in E(X) \setminus F$, 
and so $\{(u,i),(v,i)\}^{\tilde{\psi}} = \{(\psi(u),i),(\psi(v),i)\}$ is an edge of $X(F,m)$, 
for all $i \in \bZ_{m}$.  
\end{proof}

\begin{Corollary}                                               \label{cor:thcovervt}
If $X$ is a vertex-transitive graph, 
and $F \subseteq E(X)$ is a union of some edge-orbits of $X$, 
then $X(F,m)$ is vertex-transitive for every $m \ge 2$.
\end{Corollary}
\begin{proof}
The subgroup of $\Aut(X(F,m))$ generated by $\tilde{\varphi}$ 
and $\{\tilde{\psi} : \psi \in \Aut(X)\}$ acts transitively on the vertex set of $X(F,m)$.
\end{proof}

\begin{Lemma}                                                 \label{lemma:theta}
If $X$ is any graph, $F \subseteq E(X)$ and $m \ge 2$, 
and $\{u,v\}$ is any edge in $F$  such that $u$ and $v$ lie in different components of $X \setminus F$,  
then $\tilde{\theta} = \tilde{\theta}(u,v,i,j)$ is an automorphism of $X(F,m)$, for all $i,j \in \bZ_{m}$.    
\end{Lemma}
\begin{proof}
The mapping $\tilde{\theta}$ is clearly a bijection, and to prove it is an automorphism of $X(F,m)$, 
all we have to do is show that it preserves the set $E'$ of edges incident with one or more vertices not fixed 
by $\tilde{\theta}$. 
So suppose $w$ is any vertex of $X$ lying in the same component of $X \setminus F$ as $v$, 
and consider the effect of $\tilde{\theta}$ on an edge from (say) the vertex $(w,i)$ to 
a vertex $(z,k)$ in $X(F,m)$. 
If $\{w,z\} \in E(X) \setminus F$ and $k = i$, 
then $z$ lies in the same component of $X \setminus F$ as $w$ and hence in the same one as $v$, 
and therefore $\tilde{\theta}$ takes $(w,i)$ to $(w,j)$, and $(z,k) = (z,i)$ to $(z,j)$, 
which is a neighbour of $(w,j)$. 
On the other hand, if $\{w,z\} \in F$ and $k$ is arbitrary, then $\tilde{\theta}$ takes $(w,i)$ to $(w,j)$, 
and $(z,k)$ to $(z,k)$, which is a neighbour of $(w,j)$. 
The analogous things happen for edges incident with $(w,j)$ in place of $(w,i)$, 
and so the set $E'$ is preserved by $\tilde{\theta}$, as required. 
\end{proof}

Next, we note that under the assumptions of Lemma \ref{lemma:theta}, 
each automorphism $\tilde{\theta}(u,v,i,j)$ fixes the vertex $(u,k)$, for every $k \in \bZ_m$, 
and therefore the stabiliser of every such $(u,k)$ contains the automorphisms $\tilde{\theta}(u,v,i,j)$ 
for all $i,j \in \bZ_m$.

We now give a helpful example of an application of this thickened cover construction, 
to cycles of even order. 

\begin{Theorem}                                                                    \label{thm:thickenedCn}
Let $X$ be the cycle on $n$ vertices, where $n$ is even and $n > 2$, 
and let $F$ be a $1$-factor of $X$. 
Then $X(F,m)$ is vertex-transitive for all $m \ge 2$, with arc-type $m+1$ whenever $(n,m) \ne (4,2)$. 
Moreover, $X(F,m)$ is prime with respect to the Cartesian product, whenever $m \ge 2$ and $n \ne 4$.
\end{Theorem}
\begin{proof}
We may take $V(X) = \bZ_n$ and $E(X)=\{\{r,r\!+\!1 \} : \, r \in \bZ_n\}$, 
and assume without loss of generality that $F=\{\{2r ,2r\!+\!1 \} : \, r \in \bZ_n\}$.
By Lemma \ref{lemma:varphi} and Lemma \ref{lemma:theta}, we know that 
$\tilde{\varphi}$ and $\tilde{\theta}(2r,2r+1,i,j)$ are automorphisms of $X(F,m)$, 
for all $r \in \bZ_n$ and all $i,j \in \bZ_m$. 
We may also define bijections $\tilde{\rho}$  and $\tilde{\tau}$ on the vertex-set of $X(F,m)$
by 
\begin{align*}                                        
\tilde{\rho} : (u,i) \mapsto (u+2,i) 
\  \mbox{ and } \ 
\tilde{\tau}  &: (u,i) \mapsto (1-u,i) 
\quad \mbox{for all} \ u \in V(X) \ \mbox{and all} \ i \in \bZ_m.
\end{align*}
It is easy to verify that $\tilde{\rho}$ and $\tilde{\tau}$ are automorphisms of $X(F,m)$.
In fact $\tilde{\rho}$ and $\tilde{\tau}$ generate a dihedral subgroup 
of $\Aut(X(F,m))$ of order $n$ (with $n/2$ `rotations' and $n/2$ `reflections'), 
and this subgroup acts transitively on the vertices of the $i$-th layer $\{(u,i) : u \in \bZ_n\}$ of $X(F,m)$, 
for every $i \in \bZ_m$.   
It follows that the subgroup of $\Aut(X(F,m))$ generated by the 
automorphisms $\tilde{\varphi}$, $\tilde{\rho}$ and $\tilde{\tau}$ acts transitively 
on the set of all vertices of $X(F,m)$, and therefore $X(F,m)$ is vertex-transitive. 

Now let $\Delta_1$ and $\Delta_2$ be the sets of arcs associated with edges of 
the types (a) and (b) from the construction of $X(F,m)$. 
Specifically, let $\Delta_1$ be the set of arcs associated with edges of the form 
$\{(2r+1,i),(2r+2,i)\}$ for $r \in \bZ_n$ and $i \in \bZ_m$, 
and let $\Delta_2$ be the set of arcs associated with edges of the form 
$\{(2r,i),(2r+1,j)\}$ for $r \in \bZ_n$ and $i,j \in \bZ_m$. 
Note that by the thickening construction, every vertex of $X(F,m)$ is incident with 
one arc from $\Delta_1$, and with $m$ arcs from $\Delta_2$. 

All the arcs in $\Delta_1$ lie in the same orbit of $\Aut(X(F,m))$. 
In fact $\Delta_1$ is an arc-orbit of the subgroup generated 
by $\tilde{\varphi}$, $\tilde{\rho}$ and $\tilde{\tau}$, 
and here we may note that $\tilde{\tau}\tilde{\rho}$ reverses each arc $((1,i),(2,i))$, 
and so $\tilde{\rho}^{\,-r}(\tilde{\tau}\tilde{\rho})\tilde{\rho}^{\,r}$ reverses 
each arc $((2r+1,i),(2r+2,i))$ in $\Delta_1$.
Similarly, all the arcs associated with edges of the form $\{(2r,i),(2r+1,i)\}$
for $r \in \bZ_n$ and $i \in \bZ_m$ lie in the same orbit of the subgroup 
generated by $\tilde{\varphi}$, $\tilde{\rho}$ and $\tilde{\tau}$, 
and then since $\tilde{\theta}(2r,2r+1,i,j)$ interchanges the vertices $(2r,i)$ and $(2r,j)$ 
while fixing $(2r+1,i)$ and $(2r+1,j)$, 
we find that all the arcs in $\Delta_2$ lie in the same orbit of $\Aut(X(F,m))$. 

Next, we show there is no automorphism taking an arc in $\Delta_1$ to an 
arc in $\Delta_2$, unless $(n,m) = (4,2)$.  To see this, we consider the number 
of quadrangles containing a given edge.  
Every edge of the form $\{(2r,i),(2r+1,j)\}$ is contained in at least $(m-1)^2$ quadrangles, 
namely those with vertices $(2r,i)$, $(2r+1,j)$, $(2r,k)$ and $(2r+1,\ell)$, 
for given $k \in \bZ_m \setminus \{i\}$ and $\ell \in \bZ_m \setminus \{j\}$. 
On the other hand, if $n > 4$ then no edge of the form $\{(2r+1,i),(2r+2,i)\}$ is contained in a quadrangle, 
because the other neighbours of $(2r+1,i)$ and $2r+2,i)$ are all of the form $(2r,j)$ 
and $(2r+3,j)$ respectively, and no two of these are adjacent, 
while if $n = 4$, then every edge of the form $\{(2r+1,i),(2r+2,i)\}$ is contained in 
exactly $m$ quadrangles, namely those with vertices $(2r+1,i)$, $(2r+2,i)$, $(2r+3,j)$ and $(2r,j)$, 
for given $j \in \bZ_m$.  
Since $(m-1)^2 > m$ for all $m > 2$, the numbers of quadrangles are different 
when $(n,m) \ne (4,2)$. 

Hence if $(n,m) \ne (4,2)$, we find that $\Delta_1$ and $\Delta_2$ are arc-orbits of $X(F,m)$. 
Then since every vertex is incident with one arc from $\Delta_1$ 
and $m$ arcs from $\Delta_2$, the arc-type of $X(F,m)$ is $m+1$ in this case.
Finally, if $n > 4$, then by Corollary  \ref{cor:notCartesian}, 
the fact that not every edge of $X(F,m)$ is contained in a quadrangle 
implies that $X(F,m)$ is prime with respect to the Cartesian product. 
\end{proof}

In the exceptional case $(n,m) = (4,2)$, the graph $C_4(F,2)$ is isomorphic to the $3$-cube $Q_3$, 
which is arc-transitive (with arc-type $3$).  Also for every $m \ge 2$, the graph $C_4(F,m)$ is 
isomorphic to the Cartesian product of $K_{m,m}$ and $K_2$, and hence is not prime. 

Thickened covers of cycles belong to the family of cyclic Haar graphs \cite{Haar},
which are regular covering graphs over a dipole. 
Indeed the graph $C_{2k}(F,m)$ we considered in Theorem~\ref{thm:thickenedCn}
is a covering graph over a dipole with $n+1$ edges, and voltage 
assignments $1,0,m,2m,\dots,(n-1)m$ from the additive group $\bZ_{mn}$.

\smallskip
Next, we prove the following, which will be very helpful later.  

\begin{Theorem}                                                \label{thm:samearcorbit}
Let $X$ be a vertex-transitive graph, let $F$ be union of edge-orbits 
of $X$, and let $m$ be any integer with $m \ge 2$. 
Also suppose that for every edge $\{u,v\} \in F$, the vertices $u$ and $v$ lie in different 
components of $X \setminus F$, and let $(x,y)$ and $(z,w)$ be two arcs lying in the 
same arc-orbit of $X$. Then also \\[-17pt] 
\begin{itemize}
\item[{\rm (a)}] the arcs $((x,i),(y,i))$ and $((z,j),(w,j))$  lie in the same arc-orbit of $X(F,m)$
  for all $i,j,\in \bZ_m$, 
  and \\[-20pt] 
\item[{\rm (b)}] the arcs $((x,i),(y,j))$ and $((z,k),(w,\ell))$ lie in the same arc-orbit of $X(F,m)$ 
  for all $i,j,k,\ell \in \bZ_m$, when $\{x,y\} \in F$. \\[-15pt] 
\end{itemize}
\end{Theorem}
\begin{proof}
First, there exists an automorphism $\psi$ that maps  $(x,y)$ to $(z,w)$.
Now let $\tilde{\varphi}$ and $\tilde{\psi}$ be the mappings defined earlier 
in this section in \eqref{eq:varphi} and \eqref{eq:psi}.
These are automorphisms, by Lemma \ref{lemma:varphi} and Lemma \ref{lemma:psi}, 
and $\tilde{\varphi}^{\,j-i\,}\tilde{\psi}$ takes the arc $((x,i),(y,i))$ to $((z,j),(w,j))$, 
and so these two arcs lie in the same orbit of  $\Aut(X(F,m))$.

Next, suppose $\{x,y\} \in F$.  
Then also $\{z,w\} \in F$, since $\{x,y\}$ and $\{z,w\}$ are in the same edge-orbit, 
and the vertices $z$ and $w$ lie in different components of $X \setminus F$, by hypothesis.
Note that the automorphism $\tilde{\varphi}^{k-i}\tilde{\psi}$ takes $((x,i),(y,i))$ 
to $((z,k),(w,k))$,
Now let $\tilde{\theta} = \tilde{\theta}(z,w,k,\ell)$, as defined in \eqref{eq:theta}. 
This is an automorphism of $X(F,m)$, by Lemma \ref{lemma:theta}, 
and takes $((z,k),(w,k))$ to $((z,k),(w,\ell))$. 
Thus $\tilde{\varphi}^{k-i}\tilde{\psi}\tilde{\theta}$ takes $((x,i),(y,i))$  to $((z,k),(w,\ell))$, 
and so these two arcs lie in the same of  $\Aut(X(F,m))$.
\end{proof}

Note that Theorem \ref{thm:samearcorbit} cannot be pushed much further. 
For example, if $F$ is a $1$-factor in $X = C_6$,  then the graph 
$Y=C_6(F,2)$ has arc-type $2+1$, by Theorem \ref{thm:thickenedCn}. 
Now one might expect that if $\Phi_1$ is the smaller edge-orbit of $Y,$ 
then the $4$-valent graph $Y(\Phi_1,2)$ has arc-type $2+2$, 
but that does not happen: it turns out that $Y(\Phi_1,2)$ is arc-transitive, 
so has arc-type $4$.

\section{Building blocks}
\label{sec:building-blocks}

In this section we produce families of examples (and a few single examples) 
of vertex-transitive graphs with certain arc-types, which we will use as building blocks 
for the Cartesian product construction, to prove our main theorem in the final section.  
The marked partitions that occur as arc-types in these cases have a small number of summands. 
We begin with the arc-transitive case, for which there is just one summand. 

\begin{Lemma}             \label{lemma:typem}
For every integer $m \ge 2$, there exist infinitely many VT graphs 
that have arc-type $m$ and are prime with respect to the Cartesian product.
\end{Lemma}
\begin{proof}
First, when $m = 2$ we can take the cycle graphs $C_n$, for $n \ge 5$. 
These are vertex-transitive, with arc-type $2$, and taking $n > 4$ 
ensures that $C_n$ contains no 4-cycles and is therefore prime, 
by Lemma \ref{lem:notCartesian}. 

Now suppose $m \ge 3$. We construct infinitely many $m$-valent 
arc-transitive graphs, using a theorem of Macbeath \cite{Macbeath}
that tells us that for almost all positive integer triples $(m_1,m_2,m_3)$ 
with $1/m_1+1/m_2+1/m_3 < 1$, there exist infinitely many odd primes $p$ for which 
the simple group ${\rm PSL}(2,p)$ is generated by two elements $x$ and $y$ such 
that $x$, $y$ and $xy$ have orders $m_1$, $m_2$ and $m_3$, respectively.  
Here we can take $(m_1,m_2,m_3) = (2,m,m+4)$, and then for each such prime $p > m$,  
take $G = {\rm PSL}(2,p)$ and let $H$ be the cyclic subgroup of $G$ generated by $y$. 
Then $|H| = m$, and the double coset graph $\Gamma = \Gamma(G,H,x)$ 
is an arc-transitive graph of order $|G|/m = p(p-1)(p+1)/(2m)$.
This graph has valency $m$, because the stabiliser in $G$ of the arc $(H,Hx)$ 
is the cyclic subgroup $H \cap \,x^{-1}Hx$, which is trivial, since $G$ is simple. 
%
Thus $\Gamma$ has arc-type $m$. 

It remains to show that $\Gamma$ is prime.  
For the moment, suppose that $\Gamma \cong X \cp Y$ where $X$ and $Y$ are relatively prime 
non-trivial graphs.  Then by Corollary \ref{cor:producttransitive}, $X$ and $Y$ are 
vertex-transitive, and then by Theorem \ref{thm:Cartesian} we 
have $m = {\rm at}(\Gamma) = {\rm at}(X)+ {\rm at}(Y)$, which is impossible. 
Hence the prime factors of $\Gamma$ must be all the same, and so $\Gamma$ a Cartesian 
product of (say) $k$ copies of a single prime graph $X$.  But then the order of $\Gamma$ 
is $|V(X)|^k$, which is impossible unless $k = 1$, 
since the prime $p$ divides $|V(\Gamma)| = p(p-1)(p+1)/(2m)$ 
but $p^2$ does not.  Hence $\Gamma$ itself is prime. 
%
\end{proof}

At this point we remark that there are several other ways to produce infinitely many $m$-valent 
arc-transitive graphs for all $m \ge 3$.
For example, another construction uses homological covers: start with 
a given $m$-valent arc-transitive graph $X$ (such as the complete graph on $m+1$ vertices), 
and then for every sufficiently large prime $p$, construct a homological $p$-cover $\Gamma_p$ over $X$ 
with no $4$-cycles.  Then $\Gamma_p$ is also an $m$-valent arc-transitive graph, 
and is prime since it contains no 4-cycles;
see \cite{MMPP}. 

\smallskip
Next, we consider half-arc-transitive graphs. 

\begin{Lemma}                \label{lemma:type(mm)}
For every integer $m \ge 2$, there exist infinitely many VT graphs 
that have arc-type $(m+m)$ and are prime with respect to the Cartesian product.

\end{Lemma}
\begin{proof}
In 1970, Bouwer \cite{Bouwer} constructed an infinite family of vertex- and edge-transitive 
graphs of even valency, indexed by triples $(m,k,n)$ of integers such that $m,k,n \ge 2$ 
and $2^k \equiv 1$ mod $n$.  Each such graph, which we will call $B(m,k,n)$, 
has order $kn^{m-1}$ and valency $2m$.  The construction is easy, using only modular arithmetic. 
Bouwer proved that the graphs $B(m,6,9)$ are half-arc-transitive, 
thereby showing that for every even integer $2m > 2$, there exists a half-arc-transitive 
graph with valency $2m$.
Recently the first and third authors of this paper adapted Bouwer's approach to 
prove that almost all the graphs $B(m,k,n)$ are half-arc-transitive \cite{ConderZitnik}.
In particular, they showed that if $n>7$ and $k>6$ (and $2^k \equiv 1$ mod $n$),
then $X(m,k,n)$ is a half-arc-transitive graph of girth $6$, for every $m \ge 2$. 
This gives infinitely many prime graphs of type $(m+m)$, for every $m \ge 2$.
%
\end{proof}

Here we note that there are several constructions for half-arc transitive graphs.
In particular, Li and Sim used properties of projective special linear groups 
to construct infinitely many half-arc transitive graphs of every even valency 
greater than $2$ in \cite{LiSim}.
A census of 4-valent half-arc transitive graphs up to 1000 vertices is given in \cite{PSVhat}.

\begin{Lemma}                \label{lemma:typem1}
For every integer $m \ge 2$ there exist infinitely many prime VT graphs with arc-type $m+1$. 
\end{Lemma}
\begin{proof}
By Theorem \ref{thm:thickenedCn}, for every $m \ge 2$ and every even $n > 4$, 
the thickened $m$-cover of $C_n$ over a 1-factor $F$ (of $C_n$) is a prime VT graph 
with arc-type $m+1$. 
\end{proof}

In fact we will need only need one prime VT graph with arc-type $m+1$ for each $m$ 
in the proof of Theorem \ref{thm:realisable}, as we do for the next two arc-types, 
$m+(1+1)$ and $1+(m+m)$, as well. 

\begin{Lemma}                \label{lemma:typem(11)}
For every integer $m \ge 2$ there exists a prime VT graph with arc-type $m+(1+1)$. 
\end{Lemma}
\begin{proof}
Let $X$ be the graph with arc-type $1+(1+1)$  given in Figure \ref{fig:zerosym20a}.
Before proceeding, we describe some additional properties of $X$. 
First, $\Aut(X)$ is generated by the involutory automorphism $\alpha$ 
that takes $v \mapsto 21-v$ for all $v \in V(X)$, and the automorphism $\beta$ of order $4$ 
that acts as $(1, 7, 8, 2)(3, 20, 6, 9)(4, 14, 5, 15)(10, 17, 19, 12)(11, 16, 18, 13)$ 
on vertices.  In fact $\Aut(X)$ is isomorphic to the semi-direct product $C_5 \rtimes_3 C_4$,  
with normal subgroup of order $5$ generated by $\gamma = \alpha\beta^2$ ($=[\beta,\alpha]$), 
and $\beta^{-1}\gamma\beta = \gamma^3$. 
In particular, $X$ is a Cayley graph (and a GRR) for this group. 

The graph $X$ has two edge-orbits: one of size $20$ containing the edges $\{1,2\}$ and $\{1,7\}$, 
and one of size $10$ containing the edge $\{1,20\}$. 
Edges in the first orbit lie in quadrangles, while those in the second do not. 
The arc $(1,20)$ is reversed by the automorphism $\alpha$, so it lies in a self-paired 
arc-orbit, of size $20$.  
On the other hand, the arcs $(1,2)$ and $(1,7)$ lie in distinct paired arc-orbits, also of size $20$. 
(This can be seen by either considering the images of $(1,20)$ under the $20$ 
automorphisms $\beta^{i}\gamma^{j}$ (for $i \in \bZ_4$ and $j \in \bZ_5$), or by using the 
effect on $7$-cycles to show there is no automorphism that takes $(1,2)$ to $(1,7)$.) 

Now let $F$ be the smaller edge-orbit, containing the edges of the form $\{2i,2i+1\}$ 
(with the vertices considered mod $20$, so we treat $20$ as $0$), 
and let $Y$ be the thickened $m$-cover  $X(F,m)$ of $X$ over $F$.
Then $Y$ is vertex-transitive, by Corollary \ref{cor:thcovervt}, and its valency  is $m+2$.

The edges from $(1,0)$ to $(2,0)$ and $(1,0)$ to $(7,0)$ both lie in a single quadrangle, 
namely the one with vertices $(1,0)$, $(2,0)$, $(8,0)$ and $(7,0)$, 
while the edge from $(1,0)$ to $(20,0)$ lies in exactly $(m-1)^2$ quadrangles, 
namely the ones with third and fourth vertices $(1,i)$ and $(20,j)$ for any $i,j \in \bZ_m\setminus \{0\}$. 
Hence if $m > 2$, then the edges from $(1,0)$ to $(2,0)$ and $(1,0)$ to $(7,0)$ cannot lie in 
the same edge-orbit as the edge from $(1,0)$ to $(20,0)$.  
This is also true when $m = 2$, because (for example) the edges from $(1,0)$ to $(2,0)$ 
and $(1,0)$ to $(7,0)$ both lie in $16$ different $7$-cycles, while the edge from $(1,0)$ to $(20,0)$ 
lies in only $12$ different $7$-cycles.  
(This can be checked by hand or by use of {\sc Magma}.) 
Also $X \setminus F$ is a disjoint union of quadrangles 
(on vertex-sets $\{4i+1,4i+2,4i+7,4i+8\}$ for $i \in \bZ_5$), and so
Theorem \ref{thm:samearcorbit} applies.
By part (b) of Theorem \ref{thm:samearcorbit}, the edge-orbit $F$ of $X$ gives 
rise to a summand $m$ for the arc-type of $Y,$ and then by part (a), noting that $(1,2)$ 
and $(7,1)$ lie in the same arc-orbit of $X$, we find that $Y = X(F,m)$ has arc-type $m+(1+1)$ or $m+2$. 

To show that $Y$ has arc-type $m+(1+1)$, again we consider $7$-cycles. 
It is an easy exercise to show that there are exactly $4m^2$ cycles of length $7$ containing 
the edge from $(1,0)$ to $(2,0)$, namely those of the following forms: \\[-18pt] 
\begin{itemize} 
\item $( (1,0), (2, 0), (3, i), (4, i), (5, j), (6, j), (7, 0) )$, \ for any $i,j \in \bZ_m$,  \\[-20pt] 
\item $( (1,0), (2, 0), (3, i), (4, i), (18, i), (19, j), (20, j) )$, \ for any $i,j \in \bZ_m$,  \\[-20pt] 
\item $( (1,0), (2, 0), (3, i), (17, i), (18, i), (19, j), (20, j) )$, \ for any $i,j \in \bZ_m$,  \\[-20pt] 
\item $( (1,0), (2, 0), (8, 0), (9, i), (15, i), (14, j), (20, j) )$, \ for any $i,j \in \bZ_m$. \\[-18pt] 
\end{itemize}
Note that some of these can differ in only one vertex, namely in the $4$th vertex of the 
second and third forms, 
for a given pair $(i,j)$. 
Similarly, there are exactly $4m^2$ cycles of length $7$ containing 
the edge from $(1,0)$ to $(7,0)$, namely those of the following forms: \\[-18pt] 
\begin{itemize} 
\item $( (1,0), (7, 0), (6, i), (12, i), (13, j), (14, j), (20, j) )$, \ for any $i,j \in \bZ_m$,  \\[-20pt] 
\item $( (1,0), (7, 0), (6, i), (12, i), (13, j), (19, j), (20, j) )$, \ for any $i,j \in \bZ_m$,  \\[-20pt] 
\item $( (1,0), (7, 0), (6, i), (5, i), (4, j), (3, j), (2, 0) )$, \ for any $i,j \in \bZ_m$,  \\[-20pt] 
\item $( (1,0), (7, 0), (8, 0), (9, i), (15, i), (14, j), (20, j) )$, \ for any $i,j \in \bZ_m$. \\[-18pt] 
\end{itemize}
But in these cases, when two such $7$-cycles differ in only one vertex, they differ in the 
$6$th vertex (of the first and second form, for a given pair $(i,j)$).  
It follows that there can be no automorphism of $Y$ taking the arc $((1,0), (2,0))$
to the arc $((1,0), (7,0))$, and so the arc-type of $Y$ must be $m+(1+1)$. 

It remains to show that $Y = X(F,m)$ is prime. 
For this, we consider any decomposition of $Y$ into Cartesian factors, which are 
connected by Proposition~\ref{prop:connected}, and we apply Lemma \ref{lem:notCartesian1}.
The edge $\{(1,0), (2,0)\}$ lies in no quadrangle with any of the edges of the form $\{(1,i), (20,j)\}$, 
for $i \ne 0$, and it follows from part (c) of Lemma \ref{lem:notCartesian1} that all those edges must 
lie in the same factor of $Y$ as $\{(1,0), (2,0)\}$, say $U$. 
The same argument holds for the edge $\{(1,0), (7,0)\}$, and so this lies in $U$ as well.
On the other hand, every edge of the form $\{(1,0), (20,k)\}$ lies opposite 
an edge of the form $\{(1,i), (20,j)\}$ with $i \ne 0$ and $j \ne k$ in some quadrangle, 
so it lies in $U$ too.  Hence $U$ contains all $m+2$ edges incident with the vertex $(1,0)$. 
By vertex-transitivity and connectivity, all edges of $Y$ lie in $U$, and so $U = Y$. 
Thus $X(F,m)$ is prime.
\end{proof}


\begin{Lemma}                \label{lemma:type1(mm)}
For every integer $m \ge 2$ there 
exists a prime VT graph with arc-type $1+(m+m)$. 
\end{Lemma}
\begin{proof}
This is very similar to the proof of Lemma  \ref{lemma:typem(11)}. 
Again, let $X$ be the graph with arc-type $1+(1+1)$ given in Figure \ref{fig:zerosym20a}, 
but this time take $F$ to be the edge-orbit of $X$ containing the edge $\{1,2\}$ (and the edge $\{1,7\}$). 
Then the  thickened $m$-cover  $Z = X(F,m)$ of $X$ over $F$ is vertex-transitive, 
with valency $1+2m$.
Also the edge from $(1,0)$ to $(20,0)$ lies in no quadrangles, which implies 
immediately that  $Z$ is prime. 
On the other hand, 
the edge from $(1,0)$ to $(2,0)$ lies 
in $(2m-1)^2$ quadrangles, 
and so again 
the edge from $(1,0)$ to $(2,0)$ 
cannot lie in the same edge-orbit as the edge from $(1,0)$ to $(20,0)$.  
Next, $X \setminus F$ is a union of $10$ non-incident edges, 
and hence by part (a) of Theorem \ref{thm:samearcorbit}, 
we find that $Z = X(F,m)$ has arc-type $1+(m+m)$ or $2m+1$. 
Finally, as before, there are $4m^2$ cycles of length $7$ containing 
the edge from $(1,0)$ to $(2,0)$, and $4m^2$ 
containing the edge from $(1,0)$ to $(7,0)$, but when two of these cycles differ in only 
one vertex, it is in the $4$th vertex in the former case, but in the $6$th vertex in the 
latter case, and so there can be no automorphism of $Z$ taking the arc $((1,0), (2,0))$
to the arc $((1,0), (7,0))$.  Hence the arc-type of $Z$ is $1+(m+m)$. 
\end{proof}


Now we consider the marked partition $(1+1)+(1+1)$ of $4$.  
For this one, we use a quite different construction. 

\begin{Lemma}                \label{lemma:type(11)(11)}
There are infinitely many prime VT graphs with arc-type $(1+1)+(1+1)$. 
\end{Lemma}
\begin{proof}
For any prime number $p \equiv 1$ mod $6$, let $G$ be the group $C_p \rtimes_k C_6$, 
generated by two elements $a$ and $b$  of orders $6$ and $p$ such that $a^{-1}ba = b^k$, 
where $k$ is a primitive $6$th root of $1$ mod $p$. 
Also take $S = \{x,y,x^{-1},y^{-1}\}$ where $x = a$ and $y = ba^2$, 
and let $X$ be the Cayley graph $\Cay(G,S)$.

Then $X$ is a 4-valent VT graph, 
and from the natural action of $G$ by right multiplication, 
it is easy to see that all edges of the form $\{g,xg\}$ or $\{g,x^{-1}g\}$ 
lie in a single edge-orbit, as do all edges of the form $\{g,yg\}$ or $\{g,y^{-1}g\}$. 
We now show that these edge-orbits are distinct, and that each gives rise to two distinct arc-orbits, 
by proving that the stabiliser in $\Aut(X)$ of vertex $1$ is trivial. 

First, we observe that  $0 \equiv 1-(k^2)^3  \equiv (1-k^2)(1+k^{2}+k^{4})$ mod $p$, 
and then since $k^2 \not\equiv 1$ mod $p$, it follows that 
$1+k^{2}+k^{4} \equiv 0$ mod $p$, and therefore 
$$
y^3 = (ba^2)^3 = b(a^{-4}ba^{4})(a^{-2}ba^{2}) = b b^{k^{4}} b^{k^2} 
   = b^{1+k^{4}+k^{2}} = b^0 = 1.   
$$
In particular, every edge of the form $\{g,yg\}$ or $\{g,y^{-1}g\}$ lies in a 
$3$-cycle (associated with the relation $y^3 = 1$).  On the other hand, 
it is easy to see that no edge of the form $\{g,xg\}$ or $\{g,x^{-1}g\}$ lies 
in a $3$-cycle, and so $X$ has two distinct edge-orbits, and its edge-type is $2+2$. 

Similarly, we note that  $0 \equiv 1-(k^3)^2  \equiv (1-k^3)(1+k^3)$ mod $p$ 
but $k^3 \not\equiv 1$ mod $p$, so $k^3 \equiv -1$ mod $p$, and therefore 
$yx = ba^3 = a^{3}b^{-1} = a^{-3}b^{-1} = x^{-1}y^{-1}$. 
Hence the two vertices $x$ and $y^{-1}$ have two common neighbours, 
namely 1 and $yx$.  Also  
$xy = x(yx)x^{-1} = x(x^{-1}y^{-1})x^{-1} = y^{-1}x^{-1}$, 
and therefore $x^{-1}$ and $y$ have two common neighbours, 
namely $1$ and $xy$.  
Furthermore, it is an easy exercise to verify that no other two neighbours of $1$ 
have a second common neighbour.  

It follows that every automorphism $\alpha$ of $X$ that fixes the vertex $1$ 
must either fix or swap its two neighbours $y$ and $y^{-1}$, and similarly, 
must fix or swap its two neighbours $x$ and $x^{-1}$; and also if $\alpha$ swaps 
one pair, then it must also swap the other pair. 
Hence $\alpha$ either fixes all four neighbours of $1$, or induces a double transposition 
on them.  By vertex-transitivity, the same holds for any automorphism fixing a vertex $v$. 
Moreover, if the automorphism $\alpha$ fixes one of the arcs incident with the 
vertex $1$, then it fixes every neighbour $s$ of $1$, and then since it fixes the arc $(s,1)$, 
it must act trivially on the neighbourhood of $s$; then by induction and connectedness, 
$\alpha$ fixes every vertex of $X$. 

Now suppose $\Aut(X) \ne G$, or equivalently, that the stabiliser in $\Aut(X)$ of 
each vertex is non-trivial.   Now if $\beta$ and $\gamma$ are non-trivial automorphisms 
of $X$ that fix the vertex $1$, then they induce the same permutation $(x,x^{-1})(y,y^{-1})$ 
on the four neighbours of $1$, so $\beta\gamma^{-1}$ acts trivially on the neighbourhood 
of $1$ and hence is trivial, giving $\beta = \gamma$.   Hence the stabiliser of vertex $1$ 
contains a unique non-trivial automorphism, which must have order $2$. 
In particular, $|\Aut(X)| = 2|V(X)| = 2|G|$, and so $G$ is a normal subgroup 
of index $2$ in $\Aut(X)$.  
Moreover, the element of $\Aut(X)$ of order $2$ stabilising 
the vertex $1$ must normalise $G$, and hence induces an automorphism $\theta$ of $G$,  
and from what we saw earlier, $\theta$ swaps $x$ with $x^{-1}$, and $y$ with $y^{-1}$. 
Now $\theta$ takes $a = x$ to $x^{-1} = a^{-1}$, 
and $b = ya^{-2} = yx^{-2}$ to $y^{-1}x^{2} = (ba^{2})^{-1}a^{2} = a^{-2}b^{-1}a^{2} = b^{-k^2}$, 
and so $\theta$ takes $b^k$ to $(b^k)^{-k^2} = b^{-k^3} = b^{-(-1)} = b$.
But on the other hand, $b^k = a^{-1}ba$, and so $\theta$ takes $b^k$ to 
$ab^{-k^2}a^{-1} = b^{-k}$ (since $a^{-1}b^{-k}a = (b^k)^{-k} = b^{-k^2})$. 
Thus $b^{-k} = (b^k)^{\theta} = b$, and it follows that $k \equiv -1$ mod $p$, a contradiction. 

Hence no such automorphism $\theta$ of $G$ exists, and we find that $\Aut(X) = G$, 
and that $X$ has arc-type $(1+1)+(1+1)$, as required.

Finally, we show that $X$ is prime, using a similar argument to the one in the 
proof of Lemma~\ref{lemma:typem}. 
If $X \cong X_1 \cp X_2$ where $X_1$ and $X_2$ are relatively prime 
non-trivial graphs, then by Theorem \ref{thm:Cartesian} we 
have $(1+1)+(1+1) = {\rm at}(X) = {\rm at}(X_1)+ {\rm at}(X_1)$, which is impossible, 
since no VT graph has arc-type $(1+1)$. 
Hence the prime factors of $X$ must be all the same, and so $X$ is a Cartesian 
product of (say) $k$ copies of a single prime graph $Y$.  
But then $6p = |V(X)| = |V(Y)|^k$, which is impossible unless $k = 1$, 
since $p$ is a prime number congruent to $1$ mod $6$.  
Thus $X$ itself is prime. 
\end{proof}

Next, we use the first of these graphs (the one with $p = 7$) to prove the following. 

\begin{Lemma}                \label{lemma:type(mm)(11)}
For every integer $m \ge 2$, there exists a prime VT graph with arc-type $(m+m)+(1+1)$. 
\end{Lemma}
\begin{proof}
Let $X$ be the graph with arc-type $(1+1)+(1+1)$ in Figure \ref{fig:11_11}, 
which is also the graph constructed in Lemma  \ref{lemma:type(11)(11)} for $p=7$, 
and let $F$ be the edge-orbit of $X$ consisting of all the edges that are not contained in a triangle. 
(These are the edges corresponding to multiplication by the generator $x$ for $G = C_7 \rtimes C_6$.) 
Now let $Y = X(F,m)$, the thickened $m$-cover of $X$ over $F$.
Then $Y$ is vertex-transitive, by Corollary \ref{cor:thcovervt}, and its valency  is $2m+2$.
Also $X \setminus F$ is a disjoint union of triangles, so Theorem \ref{thm:samearcorbit} applies, 
and tells us that all the edges of $Y$ associated with edges of $F$ lie in the same edge-orbit, 
and all the edges of $Y$ associated with edges of $E(X) \setminus F$ lie in the same edge-orbit. 
We will show that $Y$ has arc-type  $(m+m)+(1+1)$, by proving that these edge-orbits 
are distinct, and that each gives rise to two arc-orbits. 

We do this by showing that every automorphism of $Y = X(F,m)$ induces a permutation 
of the fibres over $X$, and therefore projects to an automorphism of $X$.  
It then follows that any automorphism of $Y$ taking an arc $((v,i),(w,j))$ 
to an arc $((v',i'),(w',j'))$ gives rise to an automorphism of $X$ taking $(v,w)$ to $(v',w')$. 
Hence if $(v,w)$ and $(v',w')$ lie in different arc-orbits of $X$, then $((v,i),(w,j))$ and $((v',i'),(w',j'))$ 
lie in different arc-orbits of $Y$, for all $i,j,i',j' \in \bZ_m$. 

Observe that the graph $X \setminus F$ is a disjoint union of $14$ triangles in $X$, 
and that there are no other triangles in $X$.   
Also it is easy to see that every triangle in $Y$ is one of the $14m$ triangles of the 
form $T_i = \{(u,i),(v,i),(w,i)\}$ for some triangle $T = \{u,v,w\}$ in $X$ and some $i \in \bZ_m$, 
and that these $14m$ triangles are pairwise disjoint. 
In particular, since every automorphism takes triangles to triangles, it follows that every 
automorphism of $Y$ induces a permutation of the fibres over the edges in $E(X) \setminus F$, 
and hence also a permutation of the fibres over the edges in $F$. 
(This also implies that $Y$ is not edge-transitive, so its edge-type is $m+2$.) 

Next, consider what happens locally around a vertex $(u,i)$ of $Y$. 
This vertex lies in a unique triangle $T_i = \{(u,i),(v,i),(w,i)\}$, where $v = yu$ and $z = y^{-1}u$, 
and also lies in $m$ edges of the form $(r,j)$ and $m$ edges of the form $(s,j)$, 
for $j \in \bZ_m$, where $r = xu$ and $s = x^{-1}u$.  
The other vertices in the fibre over the vertex $u$ have the form $(u,\ell)$ for some $\ell \in \bZ_m$, 
and each of these lies at distance $2$ from $(u,i)$.  

In fact, there are $2m$ paths of length $2$ from each one to the given vertex $(u,i)$, 
namely the $m$ paths of the form $((u,\ell),(r,j),(u,i))$ for $j \in \bZ_m$, 
and the $m$ paths of the form $((u,\ell),(s,j),(u,i))$, for $j \in \bZ_m$.
On the other hand, from every other vertex at distance $2$ from the given vertex $(u,i)$ 
there are only $1$, $2$ or $m$ paths of length $2$ to $(u,i)$.  
It follows that the stabiliser in $\Aut(Y)$ of the vertex $(u,i)$ preserves the fibre over the 
vertex $(u,i)$, and therefore $\Aut(Y)$ permutes the fibres over vertices of $Y$. 

Thus $X(F,m)$ has arc-type $(m+m)+(1+1)$.

Finally, we  show that  $Y = X(F,m)$ is prime. 
If $Y \cong Y_1 \cp Y_2$ where $Y_1$ and $Y_2$ are relatively prime non-trivial 
graphs, then each $Y_i$ is vertex-transitive, and by Theorem~\ref{thm:Cartesian} 
the arc-type of $Y_1$ or $Y_2$ is $(1+1)$, which is impossible. 
Hence the prime factors of $Y$ must be all the same, and so $Y$ is a Cartesian 
product of (say) $k$ copies of a single prime graph $Z$.  
Also by part (a) of  Lemma \ref{lem:notCartesian1}, all the edges of a given 
triangle lie in the same factor, so $Z$ contains a triangle. 
But now if $k > 1$ then some subgraph of $Y$ is a Cartesian product of two triangles, 
and in the latter, every vertex lies in two distinct triangles, which does not happen in $Y$.  
Hence $k = 1$, and $Y$ itself is prime. 
\end{proof}

We use yet another construction in the next case, to produce zero-symmetric 
graphs with arc-type $1+1+1$.   Many examples of such graphs are already well 
known, but we need an infinite family of examples that are prime.  
A sub-family of the family we use below appears in \cite[p. 66]{Coxeter}.

\begin{Lemma}                \label{lemma:type111}
There are infinitely many prime VT graphs with arc-type $1+1+1$. 
\end{Lemma}
\begin{proof}
Let $G$ be the dihedral group $D_n$ of order $2n$, 
where $n$ is any integer of the form $2m-1$ where $m \ge 6$ 
(so that $n$ is odd and $n \ge 11$).   
Then $G$ is generated by two elements $x$ and $y$ satisfying $x^2 = y^n = 1$ and $xyx = y^{-1}$, 
and the elements of $G$ are uniquely expressible in the form $x^{i}y^{j}$ where $i \in \bZ_2$ and $j \in \bZ_n$. 
(In fact $G$ is the symmetry group of a regular $n$-gon, with the powers of $y$ being rotations 
and elements of the form $xy^j$ being reflections.) 

Now define $X$ as the Cayley graph $\Cay(G,\{x_1,x_2,x_3\})$, 
where $\,x_1 = x$, \ $x_2 = xy\,$ and $\,x_3 = xy^3$.  
This graph is vertex-transitive, and since the $x_i$ are involutions, 
it is $3$-valent and bipartite. 
We show that $X$ is prime and has arc-type $1+1+1$.  

The vertices at distance $2$ from the identity element are the products of two of the $x_i$, 
which are all distinct: $x_1x_2 = y$, $\,x_1x_3 = y^3$, 
$\,x_2x_1 = y^{-1}$, $\,x_2x_3 = y^2$, $\,x_3x_1 = y^{-3}\,$ and $\,x_3x_2 = y^{-2}$. 
In particular, $X$ has no $4$-cycles, and hence is prime. 
Since $X$ is bipartite, it also follows that the girth of $X$ is $6$.

Next, there are $12$ paths of length $3$ starting from the identity element, 
and but only $7$ vertices at distance $3$ from the identity element, and it is an 
easy exercise to show that the coincidences are precisely the following: 
\\[-12pt] 
\begin{center} 
\begin{tabular}{lclclclclclcl} 
$x_1x_2x_1$ & \hskip -8pt $=$ & \hskip -8pt  $x_2x_3x_2$ & \hskip -8pt $=$ & \hskip -8pt  $xy^{n-1}$,  & & 
$x_1x_2x_3$ & \hskip -8pt $=$ & \hskip -8pt  $x_2x_1x_2$ & \hskip -8pt $=$ & \hskip -8pt  $x_3x_2x_1$ 
   & \hskip -8pt $=$ & \hskip -8pt  $xy^{2}$,  
\\[+3pt]
$x_1x_3x_2$ & \hskip -8pt $=$ & \hskip -8pt  $x_2x_3x_1$ & \hskip -8pt $=$ & \hskip -8pt  $xy^{n-2}$,  & & 
\quad and & & \hskip -8pt  $x_2x_1x_3$ & \hskip -8pt $=$ & \hskip -8pt  $x_3x_1x_2$ & \hskip -8pt $=$ & \hskip -8pt  $xy^{4}$.   
\\[+3pt]
\end{tabular}
\end{center} 

In particular, the edge $\{1,x_1\}$ lies in exactly four cycles of length $6$, 
namely $(1,x_1,$ $x_2x_1,x_1x_2x_1,x_3x_2,x_2)$, 
$(1,x_1,x_2x_1,x_3x_2x_1,x_2x_3,x_3)$, $(1,x_1,x_2x_1,x_3x_2x_1,x_1x_2,x_2)$ 
and $(1,x_1,x_3x_1,x_2x_3x_1,x_3x_2,x_2)$. 
Similarly, the edge $\{1,x_2\}$ lies in exactly five $6$-cycles, and the edge $\{1,x_3\}$ lies in only three.  
These numbers are different, and it follows that the edges $\{1,x_1\}$, $\{1,x_2\}$ and $\{1,x_3\}$ lie 
in distinct arc-orbits.  

Hence the arc-type of $X$ is $1+1+1$, as claimed. 
%
%
%
%
\end{proof}

The next four lemmas deals with the remaining basic arc-types we need.

\begin{Lemma}                      \label{lemma:type1(11)}
There exist more than one prime VT graphs with arc-type $1+(1+1)$. 
\end{Lemma}
\begin{proof}
We have already observed that the zero-symmetric graph on 20 vertices 
given in Figure \ref{fig:zerosym20a} has arc-type $1+(1+1)$, 
and because not every edge is contained in a 4-cycle, it is prime by Corollary \ref{cor:notCartesian}.
Some other examples of graphs of  arc-type  $1+(1+1)$ appear in \cite[p. 39]{Coxeter}; 
these can be described with LCF-codes $[2k,2k,-2k,-2k]^m$ for $(m,k) \in \{(13,5),(17,13),(25,7),(29,17)\}$,  
and they are all prime, since they all have edges that are not contained in 4-cycles.
\end{proof}

\begin{Lemma}                      \label{lemma:type11(11)}
There exists a prime VT graph with arc-type $1+1+(1+1)$. 
\end{Lemma}
\begin{proof}
The graph on 20 vertices given in Figure \ref{fig:11_1_1} has arc-type $1+1+(1+1)$. 
Now suppose that this graph is not prime.  Then since $20$ is not a non-trivial power of any integer, 
the graph must be the Cartesian product of two smaller connected VT graphs that are 
relatively prime.  Then since there are no VT graphs with arc-type $1+1$ or $(1+1)$, 
it must be a Cartesian product of two connected VT graphs with arc-types $1$ and $1+(1+1)$. 
The former has to be $K_2$, and so the other is a VT graph of order $10$ with arc-type $1+(1+1)$. 
But no such graph exists --- in fact, the smallest VT graph with arc-type $1+(1+1)$ has 20 vertices.
Hence the given graph is prime.
\end{proof}

\begin{Lemma}                      \label{lemma:type1111}
There exists a prime VT graph with arc-type $1+1+1+1$. 
\end{Lemma}
\begin{proof}
The graph on 16 vertices given in Figure \ref{fig:1111} has arc-type $1+1+1+1$.
Also this graph cannot be the Cartesian product of two smaller connected VT graphs that are 
relatively prime, by a similar argument to the one given in the proof of Lemma~\ref{lemma:type11(11)}, 
because there is no VT graph of order $8$ with arc-type $1+1+1$. 
(The smallest VT graph with arc-type $1+1+1$ has order $18$.) 
Finally, if it is the Cartesian power of some smaller graph, then it has to be the Cartesian 
square of $C_4$ (or the Cartesian $4$th power of $K_2$, which is isomorphic to $C_4 \cp C_4$), 
but this graph is arc-transitive, with arc-type $4$.   Hence the given graph is prime. 
\end{proof}


\begin{Lemma}                      \label{lemma:type(11)(11)(11)}
There exists a prime VT graph with arc-type $(1+1)+(1+1)+(1+1)$. 
\end{Lemma}
\begin{proof}
Let $X$ be the Cayley graph $\Cay(G,S)$ for the group $G=SL(2,3)$ 
of all $2 \times 3$ matrices of determinant $1$ over $\bZ_3$, 
given by the set $S = \{x,x^{-1},y,y^{-1},xy,(xy)^{-1}\}$, where 
$$
x = \left( \begin{array}{rr} 1 & 0 \\ 1 & 1 \end{array} \right) 
\quad \hbox{and} \quad 
y = \left( \begin{array}{rr} 1 & 0 \\ -1 & 0 \end{array} \right). 
$$
These two elements generate $G$ and satisfy the relations 
$x^3 = y^4 = 1$ and $yx^{-1}=(xy)^2$, which are defining relations for $G$. 
This Cayley graph $X$ is $6$-valent, with girth $3$, and it is not difficult to show that 
its diameter is $3$.  

In the neighbourhood of the identity element $1$ in $X$, 
there is an edge between $y$ and $xy$, and a path of length $3$ from $y^{-1}$ 
to $(xy)^{-1}$ via $x$ and $x^{-1}$, but there are no other edges (between vertices of that neighbourhood). 
Also the vertices $y$ and $y^{-1}$ have another common neighbour, namely $y^2$, 
and the vertices $xy$ and $(xy)^{-1}$ have another common neighbour, namely $y^{-1}xy$ ($=(xyx)^{-1}$), 
but $y$ and $(xy)^{-1}$ have no other common neighbour, and $y^{-1}$ and $xy$ have 
no other common neighbour.  
It follows that the stabiliser in $\Aut(X)$ of the vertex $1$ either fixes all its neighbours, 
or interchanges $y$ with $xy$, and $y^{-1}$ with $(xy)^{-1}$, and $x$ with $x^{-1}$. 
But the number of neighbours of the vertex $y$ that are at distance from $x$ is $4$, 
while the number of neighbours of the vertex $xy$ that are at distance from $x^{-1}$ is only $3$, 
so the latter cannot happen, and hence the stabiliser in $\Aut(X)$ of $1$ acts trivially on 
the neighbourhood of $1$.  By vertex-transitivity, the same holds at every vertex, 
and then by induction and connectedness, it follows that the stabiliser of every vertex is trivial. 

Thus $\Aut(X) = G$, so $X$ is a GRR, and then since the edges $\{1,s\}$ and $\{1,s^{-1}\}$ 
lie in the same edge-orbit for each $s \in S$, we find that $X$ has type $(1+1)+(1+1)+(1+1)$. 

Finally, $X$ cannot be the Cartesian product of two smaller connected VT graphs that are 
relatively prime, since there are no VT graphs with arc-type $(1+1)$, 
and $X$ cannot be a Cartesian power of some smaller VT graph, since its order $24$ is 
not a non-trivial power of any integer.  Hence $X$ is prime.
\end{proof}

The graph used in the above proof is shown in Figure~\ref{fig111111}. 
\begin{figure}[h]
\centering
\includegraphics[width=5cm]{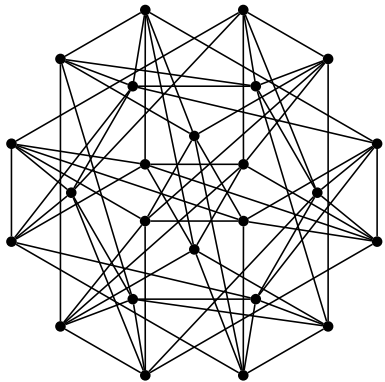}
\caption{A VT graph with arc-type $(1+1)+(1+1)+(1+1)$, on $24$ vertices}
\label{fig111111}
\end{figure}

\section{Realisability}
\label{sec:realise}

We say that a marked partition $\Pi$ is \emph{realisable} if there
exist a vertex-transitive graph with arc-type $\Pi$. 
Recall that the marked partitions  $1+1$ and $(1+1)$ are not realisable 
(as we explained in the introductory section), 
but on the other hand, we showed in Section \ref{sec:building-blocks} 
that some other marked partitions with few summands are realisable 
by infinitely many vertex-transitive graphs. 

In this final section we prove that all other marked partitions are realisable. 
We then find (as a corollary) that all standard partitions except
$1+1$ are realisable as the edge-type of a vertex-transitive graph.

\begin{Theorem}                    \label{thm:realisable}
Every  marked partition other than $(1+1)$ and $1+1$ is realisable 
as the arc-type of a vertex-transitive graph.
\end{Theorem}
%
%
\begin{proof} 
Let $\Pi = n_1 + \dots + n_t + (m_1 + m_1) + \dots + (m_s + m_s)$
be a marked partition of an integer $d \ge 2$, different from $1+1$ and $(1+1)$. 
We may assume that $n_1 \ge \dots \ge n_t$ and $m_1 \ge \dots \ge m_s$. 
%
If $d \le 4$, then we know from the examples given in Section \ref{sec:small_examples} 
that $\Pi$ is realisable, and therefore we may assume that $d \ge 5$ when necessary.

We will show how to find a VT graph with arc-type $\Pi$, 
by taking the Cartesian product of prime graphs with smaller degrees and 
simpler arc-types, chosen so that the sum of their arc-types is $\Pi$. 
To do this, we consider separately the two cases where $s = 0$ and $t = 0$, 
with a focus on the number of $n_i$ or $m_j$ that are equal to $1$, respectively, 
and then we combine these two cases in order to show how to handle all possibilities. 

\medskip\smallskip\noindent 
Case (A): \ $s = 0$, \ with $\Pi = n_1 + \dots + n_t$.  
\smallskip

Let $k$ be the number of $n_i$ that are equal to $1$, so that $n_i > 1$ for 
$1 \le i \le t-k$, and $n_t = 1$ for $t-k+1 \le i \le t$. 

If $k = 0$, then by Lemma \ref{lemma:typem} we can find $t$ pairwise non-isomorphic 
prime VT graphs with arc-types $n_1,\dots,n_t$, and by Theorem  \ref{thm:Cartesian}, 
their Cartesian product is a VT graph with  arc-type $\Pi$. 

If $k = 1$, we can take the Cartesian product of $t-2$ pairwise non-isomorphic 
prime VT graphs with arc-types $n_1,\dots,n_{t-2}$, and one prime VT graph 
with arc-type $n_{t-1}+1$, as given by Lemma \ref{lemma:typem1}, and again, 
this is a VT graph with  arc-type $\Pi$. 

If $k = 2$, we can take the Cartesian product of $t-3$ pairwise non-isomorphic 
prime VT graphs with arc-types $n_1,\dots,n_{t-3}$, one prime VT graph 
with arc-type $n_{t-2}+1$, and the graph $K_2$. 

Finally, if $k \ge 3$, we can take the Cartesian product of $t-k$ pairwise non-isomorphic  
prime VT graphs with arc-types $n_1,\dots,n_{t-k}$, plus  \\[-21pt] 
\begin{itemize} 
\item[(i)] $k/3$ pairwise non-isomorphic prime VT graphs with arc-type $1+1+1$ taken 
from Lemma \ref{lemma:type111}, when $k \equiv 0$ mod $3$, or \\[-24pt] 
\item[(ii)] one prime VT graph of type $1+1+1+1$ from Lemma \ref{lemma:type1111}, 
and $(k-4)/3$ pairwise non-isomorphic  prime VT graphs with arc-type $1+1+1$, 
when $k \equiv 1$ mod $3$, or \\[-24pt] 
\item[(iii)] one copy of $K_2$, and one prime VT graph of type $1+1+1+1$, and $(k-5)/3$ pairwise 
non-isomorphic  prime VT graphs with arc-type $1+1+1$, when $k \equiv 2$ mod $3$.  \\[-22pt] 
\end{itemize}

\medskip\smallskip\noindent 
Case (B): \ $t = 0$, \ with $\Pi = (m_1 + m_1) + \dots + (m_s + m_s)$.  
\smallskip

Let $\ell$ be the number of $m_j$ that are equal to $1$, so that $m_j > 1$ for 
$1 \le i \le s-\ell$, and $m_j= 1$ for $s-\ell+1 \le i \le s$. 

If $\ell = 0$, then by Lemma \ref{lemma:type(mm)} we can find $s$ pairwise non-isomorphic 
VT graphs with arc-types $(m_1 + m_1),\dots,(m_s + m_s)$, and then their Cartesian product is 
a VT graph with  arc-type $\Pi$. 

If $\ell = 1$, we can take the Cartesian product of $s-2$ pairwise non-isomorphic VT graphs with 
arc-types $(m_1 + m_1),\dots,(m_{s-2} + m_{s-2})$, and one prime VT graph with 
arc-type $(m_{s-1} + m_{s-1})+(1+1)$ from Lemma \ref{lemma:type(mm)(11)}. 

Finally, if $\ell \ge 2$, we can take the Cartesian product of $s-\ell$ pairwise non-isomorphic 
prime VT graphs with arc-types $(m_1 + m_1),\dots,(m_{s-\ell} + m_{s-\ell})$, 
plus  \\[-22pt] 
\begin{itemize} 
\item[(i)] $\ell/2$ pairwise non-isomorphic  prime VT graphs with arc-type $(1+1)+(1+1)$ 
from Lemma \ref{lemma:type(11)(11)}, when $\ell$ is even, or \\[-24pt] 
\item[(ii)] one prime VT graph of type $(1+1)+(1+1)+(1+1)$ from Lemma \ref{lemma:type(11)(11)(11)}, 
and $(\ell-3)/2$ pairwise non-isomorphic  prime VT graphs with arc-type $(1+1)+(1+1)$, 
when $\ell$ is odd.  \\[-24pt] 
\end{itemize} 

\medskip\smallskip\noindent 
Case (C): \ $s > 0$ and $t > 0$. 
\smallskip

In this case, we can write $\Pi$ as the sum of the marked partitions $\Pi_1 = n_1 + \dots + n_t$ 
and $\Pi_2 = (m_1 + m_1) + \dots + (m_s + m_s)$, and we can deal with most possibilities 
by simply taking a Cartesian product of a VT graph $X_1$ with arc-type $\Pi_1$ 
and a VT graph $X_2$ with arc-type $\Pi_2$. 
Note that case (A) uses the prime VT graphs produced by Lemmas \ref{lemma:typem}, 
\ref{lemma:typem1}, \ref{lemma:type111} and  \ref{lemma:type1111}, plus the graph $K_2$,
while case (B) uses the prime VT graphs produced by Lemmas \ref{lemma:type(mm)}, 
\ref{lemma:type(mm)(11)}, \ref{lemma:type(11)(11)} and \ref{lemma:type(11)(11)(11)}. 
These prime graphs can be chosen to be pairwise non-isomorphic, 
and hence pairwise relatively prime, in which case $X_1$ and $X_2$ are relatively prime, 
and therefore $X_1 \cp X_2$ has arc-type $\Pi_1+\Pi_2 = \Pi$. 

All that remains for us to do is to deal with the exceptional situations, 
namely those where the sum of the $n_i$ or the sum of the $m_j$ is so small that 
no suitable candidate can be found for $X_1$ or $X_2$.  
There are two exceptional possibilities for $\Pi_1$ not covered in case (A), 
namely $1$ and $1+1$, and just one for $\Pi_2$ in case (B), namely $(1+1)$.  

If $\Pi_1 = 1$, then we can take a Cartesian product of $K_2$ with a VT graph 
produced in case (B), since we are assuming that $d \ge 5$.  

If $\Pi_1 = 1+1$, then there are two sub-cases to consider, depending on the number $\ell$ 
of terms $m_j$ that are equal to $1$. 
If $\ell < s$, then we can adapt the approach taken in case (B) by replacing the 
prime VT graph of type $(m_1 + m_1)$ by a prime VT graph of type $1+(m_1 + m_1)$ 
from Lemma \ref{lemma:type1(mm)}, and then also add a single copy of $K_2$ as above. 
On the other hand, if $\ell = s$, so that $\Pi_2$ is a sum of $s$ terms of the form $(1+1)$, 
then we take \\[-22pt] 
\begin{itemize} 
\item[(i)] two non-isomorphic  prime VT graphs with arc-type $1+(1+1)$ from Lemma \ref{lemma:type1(11)}, 
when $s = 2$, or \\[-24pt] 
\item[(ii)] a prime VT graph of type $1+1+(1+1)$ from Lemma \ref{lemma:type11(11)},  
and a $2(s-1)$-valent VT graph of type $(1+1)+\dots+(1+1)$ as found in case (B) when $s \ge 3$. \\[-22pt] 
\end{itemize} 

Finally, if $\Pi_2 = (1+1)$, again there are two sub-cases to consider, this time depending 
on the number $k$ of terms $n_i$ that are equal to $1$. 
If $k < t$, then we can adapt the approach taken in case (A) by replacing the prime VT graph 
of type $n_1$ by a prime VT graph of type $n_1+(1 + 1)$ from Lemma \ref{lemma:typem(11)}. 
On the other hand, if $k = t$, so that $\Pi_1$ is a sum of $t$ terms all equal to $1$, 
then we take \\[-22pt] 
\begin{itemize} 
\item[(i)] a single copy of $K_2$ and a prime VT graph with arc-type $1+1+(1+1)$ from Lemma \ref{lemma:type11(11)}, when $t = 3$, or \\[-24pt] 
\item[(ii)] a prime VT graph of type $1+(1+1)$ and a $(t-1)$-valent VT graph of type $1+\dots+1$ 
as found in case (A) when $t \ge 4$. \\[-22pt] 
\end{itemize} 

This completes the proof.
\end{proof}

\begin{Corollary}
With the exception of $1+1$ $($for the integer $2)$, 
every standard partition of a positive integer is realisable as the edge-type 
of a vertex-transitive graph. 

\end{Corollary}
\begin{proof}
This follows easily from Theorem \ref{thm:realisable}. 
In fact, every such partition $n_1 +\dots + n_t$ (except $1+1$) occurs as 
both the edge-type and the arc-type of some VT graph 
with the property that all of its arc-orbits are self-paired. 
\end{proof}

\section*{Acknowledgements}

We would like to thank  Primo\v{z} Poto\v{c}nik and Gabriel Verret for fruitful discussions 
about vertex-transitive graphs, and Matja\v{z} Konvalinka and Marko Petkov\v{s}ek for their 
advice about partitions.
Also we acknowledge the use of {\sc Magma} \cite{Magma} in constructing examples
of graphs with a given arc-type and small valency, as well as in testing various cases. 
In the course of experimenting, Mathematica and Sage \cite{Sage} were used as well.

This work was supported in part by the N.Z. Marsden Fund (via grant UOA1323), 
the ARRS (via grant  P1--0294), the European Science Foundation (Eurocores Eurogiga, GReGAS (N1--0011))
and the Air Force Office of Scientific Research, Air Force Material Command, USAF (under Award No. FA9550-14-1-0096).


\bigskip

\noindent
{\sc Marston D.E. Conder,} \\  
Department of Mathematics, University of Auckland,\\
Private Bag 92019, Auckland 1142, {\sc New Zealand}\\
Email: \texttt{m.conder@auckland.ac.nz}\\

\noindent
{\sc Toma\v z Pisanski,} \\ 
UP IAM and UP FAMNIT, University of Primorska, \\
Muzejski trg 2, 6000 Koper, {\sc Slovenia}\\
Email: \texttt{tomaz.pisanski@fmf.uni-lj.si}\\

\noindent
{\sc Arjana \v Zitnik}\\ 
Faculty for Mathematics and Physics, University of Ljubljana, and IMFM,\\
Jadranska 19, 1000 Ljubljana, {\sc Slovenia}\\
Email: \texttt{arjana.zitnik@fmf.uni-lj.si}

\end{document}